\documentclass[11pt]{amsart} % {article}

\pdfoutput=1

\usepackage[utf8]{inputenc}
\usepackage[T1]{fontenc}
\usepackage[british]{babel}
\usepackage{graphicx}
\usepackage[sectionbib,numbers]{natbib}
\usepackage{amstext,amsthm,amsmath,amssymb,mathtools,bbm,mathrsfs,mathtools}
\usepackage{url}
\usepackage{geometry}
\usepackage[hidelinks]{hyperref}
\usepackage{graphicx}
\usepackage{enumerate}
\usepackage{xcolor}

\usepackage{tikz}

\usepackage{todonotes} % allows to add todo notes in the margins, usage: \todo{...}
\usepackage[normalem]{ulem}
\usepackage{hyperref}

\newtheorem{proposition}{Proposition}[section]
\newtheorem{theorem}[proposition]{Theorem}
\newtheorem{corollary}[proposition]{Corollary}
\newtheorem{lemma}[proposition]{Lemma}
\theoremstyle{definition}
\newtheorem{remark}[proposition]{Remark}

\newcommand{\N}{\mathbb{N}}
\newcommand{\bN}{\mathbb{N}}
\newcommand{\bZ}{\mathbb{Z}}
\newcommand{\Z}{\mathbb{Z}}
\newcommand{\R}{\mathbb{R}}
\newcommand{\C}{\mathbb{C}}
\newcommand{\bR}{\mathbb{R}}
\renewcommand{\Pr}{\mathbb{P}}
\newcommand{\bE}{\mathbb{E}}

\newcommand{\abs}[1]{\lvert#1\rvert} %Definition Betragsstriche

\newcommand{\norm}[1]{\lVert#1\rVert} % Definition Normstriche

\newcommand{\ind}[1]{\mathbbm{1}_{\{#1\}}} %Definition Indikatorfunktion
\newcommand{\indset}[1]{\mathbbm{1}_{#1}} %Definition Indikatorfunktion

\newcommand\restr[2]{\ensuremath{\left.#1\right|_{#2}}} % Definition Einschraenkung

\newcommand{\wt}{\widetilde}

\newcommand{\joint}{{\mathrm{joint}}}

\newcommand{\indi}{{\mathrm{ind}}}

\newcommand{\bP}{{\mathbb{P}}}
\newcommand{\compl}{\mathsf{C}}

\newcommand{\Lt}{L_{\mathrm{t}}}
\newcommand{\Ls}{L_{\mathrm{s}}}

\newcommand{\hittingLemma}{3.20}
\newcommand{\regenerationConstr}{4.5}

\newcommand{\abstractPropRem}{4.9}
\newcommand{\separationlemma}{4.12}

\numberwithin{equation}{section}

\title[Pair coalescence times of ancestral lineages of logistic branching random walks]{%
  Pair coalescence times of ancestral lineages of two-dimensional logistic branching random walks
}
\author{Matthias Birkner, Andrej Depperschmidt, Timo Schl\"{u}ter}
\date{\today}

\dedicatory{We dedicate this paper to Alison M.\ Etheridge on the occasion of her 60th birthday.}

\begin{document}
\maketitle

\begin{abstract}
  Consider two ancestral lineages sampled from a system of
  two-dimensional branching random walks with logistic regulation in
  the stationary regime. We study the asymptotics of their coalescence
  time for large initial separation and find that it agrees with well
  known results for a suitably scaled two-dimensional stepping stone
  model and also with Mal\'ecot's continuous-space approximation for the probability of identity by descent as a function of sampling distance.
  This can be viewed as a justification for the replacement of locally
  fluctuating population sizes by fixed effective sizes. Our main tool
  is a joint regeneration construction for the spatial embeddings of
  the two ancestral lineages.
\end{abstract}

\tableofcontents

\section{Introduction and main result}

\subsection{The Wright-Mal\'ecot formula}
Consider a population of a certain species which lives, reproduces and
evolves in a two-dimensional space.  We assume that the population is
composed of individuals with possibly different (but neutral) genetic
types, and that offspring disperse (only) locally around their
parent's location. Imagine that environmental conditions are -- at
least approximately -- homogeneous in space and time, that the
population has been around for quite long and has reached some sort of
``equilibrium'' with respect to its environment, and that the habitat
is very large compared to the distance a typical individual can travel
during their lifetime. Then it seems a reasonable mathematical
abstraction to assume that the state of the population can be
described by some spatial random process which is (in distribution)
shift-invariant in space and time and that one observes at any given
time this process in its equilibrium state.  For simplicity, we will
assume in the following that individuals are haploid, and that
mutations occur according to the so-called infinite alleles model,
i.e.\ a mutation will always generate a novel type.

Imagine that we sample two individuals from this population at a given
time, say one from the origin and another one at position $x$
($\in \R^2$, say). What is the probability $\phi(x)$ that the two
sampled individuals have the same genetic type?

For such a situation, assuming that space is continuous and given by
the two-dimensional plane $\R^2$, Gustave Mal\'ecot
\cite{Malecot:1948, Malecot:1969} (and also Sewall Wright \cite{Wright1943, Wright1946}
who formulated a series representation of the same term) gave a
remarkable approximation formula to answer this question, namely
\begin{equation}
  \label{eq:MF}
  \phi(x) \approx \frac{1}{2\pi \sigma^2 \delta + K_0(\sqrt{2\mu} \, \kappa/\sigma)}
  K_0\big( \sqrt{2\mu} \, \norm{x}/\sigma \big)
  \quad \text{ for } \norm{x} \ge \kappa.
\end{equation}
Here, the parameters have the following interpretations: $\delta>0$ is
the local population density (i.e., in a region $A$ one should
find on average $\delta |A|$ many individuals), $\sigma>0$ is the
standard deviation of the spatial displacement between a typical
individual and its parent, % (viewed backwards in time),
$\mu>0$ is the mutation probability per
generation and $\kappa > 0$ is a ``local scale'' parameter;
% \notiz{see discussion below in ...}.
the intuitive idea is that the approximation breaks down for distances
smaller than $\kappa$.  Finally, $K_0(\cdot)$ is the modified Bessel
function of the second kind of order $0$ (see
Remark~\ref{rem:modBessel} below, where we recall relevant
properties).

G.~Mal\'ecot derived the formula \eqref{eq:MF} for $\phi(x)$
in the 1940ies by considering the recursion formula
\begin{equation}
  \label{eq:MFrecursion}
  \begin{split}
    \phi({y})
    & = (1-\mu)^2 \bigg( \frac{1 - \phi({0})}{\delta}
    \int_{\R^2} g_{\sigma^2}({y} - {z}) g_{\sigma^2}({z}) \,d{z}  \\
    & \hspace{5.75em} 
    + \int_{\R^2 \times \R^2} g_{\sigma^2}({z}) g_{\sigma^2}({z}') \phi({y}+{z}-{z}') \,d{z} d{z}'
    \Big)
  \end{split}
\end{equation}
($g_{\sigma^2}(\cdot)$ is the twodimensional normal density with covariance $\sigma^2$ times the identity matrix),
then taking the Fourier transform and (formally) inverting it.
Briefly, the rationale behind \eqref{eq:MFrecursion} is a backwards in time
argument. Imagine sampling two different individuals at spatial separation $y \neq 0$ and decompose
the event that they have the same type according to their ancestry one generation
ago. Neither must have experienced a mutation, which explains the factor $(1-\mu)^2$
on the right-hand side. Next, imagine that the spatial displacement of these two individuals
from their (respective) ancestors is described by a centred isotropic Gaussian displacement
with variance $\sigma^2$. The second integral on the right-hand side refers to the case that
these ancestors lived at different positions (and are then necessarily distinct, with a new
spatial separation) whereas the first integral comes from the case when the two ancestors
lived at the same position (and are then with probability $1/\delta$ in fact the same individual). 

This is arguably a remarkably early instance of the `modern'
retrospective viewpoint in mathematical population genetics which has
become very prominent since the introduction of the Kingman
coalescent and its
many ramifications \cite{Kingman1982} (and also \cite{Hudson1983-1, Hudson1983-2, Tajima1983}
in the biology literature).
It is, however, not entirely rigorous. In fact,
%it is known that
no stochastic population model exists in two-dimensional
continuous space $\R^2$ which has both a non-trivial and homogeneous
stationary distribution as well as a dynamics of embedded ancestral
lineages literally compatible with \eqref{eq:MFrecursion}.
The situation is different on the discrete space $\Z^2$, where one can
meaningfully condition on constant local population sizes/densities,
leading to the so-called stepping stone models (see also Remarks~\ref{rem:steppingstone}
and \ref{rem:MFproblems} below).

The Wright-Mal\'ecot formula \eqref{eq:MF} and its surrounding
philosophy (as well as its difficulties) are nicely explained e.g.\ in
\cite{BartonDepaulisEtheridge:2002} and also in \cite{Etheridge2006,
  Etheridge:EBP23}. See also Remark~\ref{rem:MFproblems} below, where
we discuss some aspects relevant to the present study.

\subsection{Logistic branching random walks}
\label{sect:log branching random walks}
As mentioned above (and discussed in more detail in
Remarks~\ref{rem:MFproblems}, \ref{rem:steppingstone} below), a
popular way of modelling spatially distributed populations is
to %consider discrete space (decomposing the habitat, possibly
  % somwhat artificially, into discrete ``demes'')
decompose the habitat, possibly somewhat artificially, into discrete
``demes'', whose population sizes or densities are exogeneously given
or fixed, maybe varying in space and time as a deterministic function.
Our aim in this paper is to derive an asymptotic analogue of
\eqref{eq:MF} for a population model which has a non-trivial
equilibrium in two-dimensional space with fluctuating local population
sizes due to randomness in the reproduction and an endogeneous
feed-back mechanism.  We consider one prototypical, discrete model
(both space, time and population numbers are discrete), for which our
approach is technically easier than in the continuum case. We believe
however that with more technical effort, continuous models could be
studied in a similar fashion, see also Section~\ref{sect:outlook}
below.

Specifically, we consider the following version of logistic branching
random walks (LBRW), which was studied in
\cite{BirknerDepperschmidt2007}. Let
$p=(p_{xy})_{x,y \in \bZ^d} =(p_{y-x})_{x,y \in \bZ^d}$ be a symmetric
aperiodic stochastic kernel on $\bZ^2$ with finite range $R_p \ge 1$,
$\lambda=(\lambda_{xy})_{x,y \in \bZ^d}$ a non-negative symmetric
kernel on $\bZ^2$ with finite range $R_\lambda$ and $m \in (1,3)$. We assume that $p$ is the transition kernel of an aperiodic and irreducible random walk on $\bZ^2$. For
a configuration $\xi \in \R_+^{\bZ^d}$ and $x \in \bZ^d$ we define
\begin{align}
  \label{def:f}
  f(x; \xi) \coloneqq \xi(x) \Big( m - \sum_{z \in \bZ^d} \lambda_{xz} \xi(z)  \Big)^+.
\end{align}
A configuration of the population $\eta$ is a (random) element of
$\bN_0^{\bZ^d\times \bZ}$ where $\eta_n(x)$ is the number of particles
at position $x \in \Z^d$ in generation $n\in\bZ$. The population
evolves in the following way: given $\eta_n$, the state of the
population at time $n$, each particle at $x$ has
$\mathrm{Pois}\big( \big(m - \sum_z \lambda_{z-x}
\eta_n(z)\big)^+\big)$ many offspring. We see from the parameter
$\big(m - \sum_z \lambda_{z-x}\eta_n(z)\big)^+$ that $\lambda$
introduces local competition and each individual at $z$ reduces the
average reproductive success of the focal individuals at $x$ by
$\lambda_{z-x}$. Thus we also call $\lambda$ the \emph{competition
kernel}. Lastly the offspring move a random walk step from their
parent's position independently of each other according to the
transition kernel $p$. Therefore, given $\eta_n$, we obtain by the
superposition properties of independent Poisson random variables 
the following generation via
\begin{align}
        \label{eq:law of eta in log branching}
        \eta_{n+1} (x)  \sim  \mathrm{Pois}\Bigl( \sum_{y \in \Z^d}
        p_{yx} f(y;\eta_{n}) \Bigr), \quad
        \text{ independently for } x \in \Z^d.
\end{align}
In summary, $(\eta_n)$ is a spatial population model with local
density-dependent feedback, i.e.\ the offspring distribution is
supercritical when there are few individuals in the vicinity and
subcritical whenever there are too many. In general this system is not
attractive, i.e., adding particles to the initial condition can
stochastically decrease the population at later times. This is owed
to the fact that the competition kernel is non-local as well as to the
fact that by discreteness of time, all sites are updated
simultaneously.

%\notiz{Interpretation.
%
%  Given $\eta_n$, each particle at $x$ has
%  $\mathrm{Poisson}\big( \big(m - \sum_z \lambda_{z-x}
%  \eta_n(z)\big)^+\big)$ offspring.
%
%  %% Interpretation as local competition:
%  Local competition: Each individual at $z$ reduces average
%  reproductive success of focal ind.\ at $x$ by $\lambda_{z-x}$.
%
%  Children take an independent random walk step to $y$ with
%  probability $p_{y-x}$.
%}
%
%\notiz{$(\eta_n)$ is a spatial population model with local
%  density-dependent feedback: Offspring distribution supercritical
%  when there are few neighbours, subcritical when there are many
%  neighbours.
%
%  The system is in general \emph{not} attractive.}

\begin{theorem}[\protect{\cite[Theorem~3 and Corollary~4]{BirknerDepperschmidt2007}}]
        \label{thm:eta_survival}
  Assume $m \in (1,3)$, $0< \lambda_0 \ll 1$, $\lambda_z \ll \lambda_0$ for $z \neq 0$.
 Then $(\eta_n)$ survives for all time globally and locally
  with positive probability for any non-trivial initial
  condition $\eta_0$.

  Given survival, $\eta_n$ converges in distribution to its unique
  non-trivial equilibrium $\overline{\nu}$.
\end{theorem}
The restriction on $m$ stems from the fact that the proof in
\cite{BirknerDepperschmidt2007} requires that the logistic map
$x \mapsto x(m-\lambda x)^+$ has a (unique) strictly positive stable
fixed point (namely, $(m-1)/\lambda$).
% two fix points $0$ and $(m-1)/\lambda$ which, for $m\ge 3$, are both not stable.
For more discussion see \cite{BirknerDepperschmidt2007}.

\subsection{Ancestral lineages of two samples from LBRW}
\label{sect:Ancestral lineages}
Consider the stationary process $(\eta^{\mathrm{stat}}_n)_{n \in \Z}$
with $\mathcal{L}(\eta^{\mathrm{stat}}_n) = \overline{\nu}$ for all
$n \in \Z$, enriched with ``enough book-keeping'' to follow ancestries
of particles.  Since we have discrete particles, it is in principle
straightforward, though notationally cumbersome, to keep track of each
particle's ancestry.  We will not make this book-keeping explicit in
the following but keep in mind that it is in principle there (see also
Chapter~4 in \cite{Depperschmidt08} or Section~4 in
\cite{BirknerCernyDepperschmidt2016}).
\smallskip

Let $x \in \Z^d$, $x \neq 0$, sample one individual at random from $\eta^{\mathrm{stat}}_0$
at $0$ and one from $x$. (We implicitly condition on $\eta^{\mathrm{stat}}_0(0)>0$ and
$\eta^{\mathrm{stat}}_0(x)>0$ so that such sampling is possible. In the parameter
regime we consider, $\eta^{\mathrm{stat}}_0$ will typically have a very large population
density anyway, so the effect of this conditioning is extremely mild.) We denote \\[1ex]
\rule{1em}{0em} $X=(X_k)_{k \in \N_0}$,
$X_k $ = \parbox[t]{0.8\textwidth}{position of ancestor $k$ generations in the past
  of particle sampled at $0$}\\[1ex]
\rule{1em}{0em} $X'=(X'_k)_{k \in \N_0}$,
$X'_k $ = \parbox[t]{0.8\textwidth}{position of ancestor $k$ generations in the past
  of particle sampled at $x$} \\[1ex]
%\smallskip
%
\noindent
and let $\tau_{\mathrm{coal}}$ be the time (generations backwards in
time) to the most recent common ancestor of the two sampled particles. Furthermore 
we define $\bP_{x,x'}$ to be the measure with $X$ started at $x$ and $X'$ at $x'$, i.e.\ 
we condition on $X_0=x$ and $X'_0=x'$.
\smallskip

The pair $(X, X')$ can be interpreted as two delayed coalescing random walks in a dynamic random environment (which is generated by the time reversal of the population's space-time occupation field $\eta$), see  %Section~\ref{sect:ResultsfromBDS24} (especially
Section~\ref{subsect:Ancestral lineages in LBRW} and Appendix~\ref{sect:Coalescing probabilities} for details. Random walks in dynamic random environments (RWDRE) are a very active research topic, see e.g.\ the discussion and references in \cite{BirknerDepperschmidtSchlueter2024, BirknerCernyDepperschmidt2016}. However, the possibility of coalescence, which is our focus here, is a scenario that is typically not considered in the context of RWDRE.

Note that the dynamics of $(X, X')$ when averaging over $\eta$ is not Markovian (intuitively, at a given time, past behaviour of $(X, X')$ contains then information on the population densities the walks will experience in the future). When fixing $\eta$ (i.e., in the context of random walks in random environments considering the ``quenched law''), the dynamics is in Markovian (see Appendix~\ref{sect:Coalescing probabilities}) but the transition probabilities are space-time inhomogeneous and depend in a relatively complicated way on $(\eta_n(x), n \in \bZ_-, x \in \bZ^2)$.

\subsection{Asymptotics for pair coalescence times and the probability
  of identity by descent}
With all other parameters fixed, \eqref{eq:MF} yields $\phi(x) \to 0$
as $\norm{x}\to\infty$.  On the other hand, for fixed $x$ (and
$\delta, \sigma, \kappa$ fixed as well), \eqref{eq:MF} gives
$\phi(x) \to 1$ as $\mu \to 0$. In applications, we might be
interested in the population's genetics structure over large spatial
scales and the mutation rate at the locus under study might be very
small. Thus, it is not unnatural to assume that $\norm{x}$ is very large
and at the same time $\mu$, though positive, is tiny.

In such a situation, interesting structure can arise from scaling.
Assume that $x = Ny$ with $y \in \R^2 \setminus \{0\}$ and
$\mu=\mu_N = mN^{-2\gamma}$ with $m \in (0,\infty)$ and $\gamma \ge 1$
are scaled with $N\to\infty$ and we keep the other parameters
$\delta$, $\sigma$ and $\kappa$ fixed: we sample a pair with a large
separation and the mutation rate is very small (proportional to a
certain negative power of the separation).  In this scenario
\eqref{eq:MF} yields
\begin{align}
  \label{eq:MF.scaling}
  \frac{K_0\big( \sqrt{2 m} \norm{y} N^{1-\gamma} /\sigma \big)}{2\pi \sigma^2 \delta + K_0\big(\sqrt{2m} \, \kappa N^{-\gamma} /\sigma\big)}
  \sim \frac{-\log\big( \sqrt{2 m} \norm{y} N^{1-\gamma} /\sigma \big)}{-\log\big(\sqrt{2m} \, \kappa N^{-\gamma} /\sigma\big)}
  \, \mathop{\longrightarrow}_{N\to\infty} \, 1 - \frac{1}{\gamma}
\end{align}
where we used the asymptotics of $K_0(\cdot)$ near $0$ from \eqref{eq:BesselK0.bdry}.

\begin{remark}
	\label{rem:defn_prop_measures}
	To lessen the overall notation we will often write $\bP_{Nx}=\bP_{0,Nx}$ when we explicitly consider a particle starting at the origin and the other starting at $Nx$. Note that $N$ is just a scaling parameter while $x\in \bR^2\setminus\{0\}$ is some direction and we always assume $Nx$ to be in $\bZ^2$, e.g.\ by rounding. This small ``error'' has no influence on the results and carrying it through all calculations would unnecessarily clutter the proofs. Later, in the proofs, the calculations often only depend on the distance of the particles and not their actual positions. In these cases we also write $\bP_R$ if the particles start at distance $R>0$ to highlight this dependence.
\end{remark}
%%\notiz{[For following theorem note Proposition~\ref{prop:limi_meeting_time}, replace $\tau_{\mathrm{meet}}$ by $\tau_{\mathrm{coal}}$ in \eqref{eq:limit_meeting_time}]}
\begin{theorem}
  \label{thm:Pct}
  Let $d=2$, assume $m \in (1,3)$, $0< \lambda_0 \ll 1$, $\lambda_z \ll \lambda_0$ for $z \neq 0$.
  We have
  \begin{equation}
    \label{eq:limit_coalescence_time}
    \lim\limits_{N\to \infty} \bP_{Nx}\big( \tau_{\mathrm{coal}} >N^{2\gamma}  \big) = \frac{1}{\gamma}
    \quad \text{ for } \gamma \ge 1.
  \end{equation}
  ($\bP_{Nx}$ refers to sampling at separation $Nx$, $x \neq 0$.)
\end{theorem}

  % Furthermore,
  % \eqref{eq:limit_coalescence_time} implies for $\gamma \ge 1$
  % \begin{equation}
  %   \lim\limits_{N\to \infty} \bE_{Nx}\big[ \exp(- N^{-2\gamma} \tau_{\mathrm{coal}}) \big] = 1 - \frac{1}{\gamma}.
  % \end{equation}
  % note: with $\mu_N = m N^{-2\gamma}$, $\phi_N(Nx) = \bE_{Nx}\big[ (1-\mu_N)^{2 \tau_{\mathrm{coal}}} \big]
  %   \sim \bE_{Nx}\big[ \exp(- \mu_N \tau_{\mathrm{coal}}) \big] $

\begin{corollary}[An asymptotic analogue of Mal\'ecot's formula for LBRW]
  \label{cor:Malecot}
  Under the assumptions of Theorem~\ref{thm:Pct}, for $\mu_N = m N^{-2\gamma}$
  with $\gamma \ge 1$, $m \in (0,\infty)$ we have
  \begin{align}
    \label{eq:LT_coalescence_time}
    \phi_N(Nx) =
    \bE_{Nx}\big[ (1-\mu_N)^{2 \tau_{\mathrm{coal}}} \big]
    \mathop{\longrightarrow}_{N\to\infty} 1 - \frac{1}{\gamma} .
  \end{align}
\end{corollary}
Our proof of Theorem~\ref{thm:Pct} uses a joint regeneration construction for the spatial embedding $(X,X')$ of the ancestral lineages of the two samples from Section~\ref{sect:Ancestral lineages}. This construction was introduced in \cite{BirknerDepperschmidtSchlueter2024} as a key technical tool to prove a quenched central limit theorem (literally, in a slightly different form since coalescence was not considered there). We recall relevant concepts and results from \cite{BirknerDepperschmidtSchlueter2024} in Section~\ref{sect:ResultsfromBDS24}.

\subsection{Remarks and discussion}

\begin{remark}[Modified Bessel function]
  \label{rem:modBessel}
  For convenience, we recall here properties of Bessel functions
  which are relevant for this text, see also e.g.\ \cite[\href{https://dlmf.nist.gov/10}{Ch.~10}]{NIST:DLMF}. For $\nu \in \C$,
  the modified Bessel equation of order $\nu$ is \cite[\href{https://dlmf.nist.gov/10.25.E1}{10.25.1}]{NIST:DLMF}
  \begin{equation}
    \label{eq:BesselODE}
    z^2 \frac{d^2}{dz^2}w(z) + z \frac{d}{dz}w(z) - (z^2+\nu^2)w(z) = 0.
  \end{equation}
  There are two `standard' solutions to \eqref{eq:BesselODE}, $I_\nu(\cdot)$ and $K_\nu(\cdot)$, the modified
  Bessel functions of the first and of the second kind, respectively, of order $\nu$.

  We focus here on the case $\nu = 0$ and we consider only real, non-negative arguments. $K_0(t)$
  then solves
  \begin{equation}
    \label{eq:BesselODE.0}
    \frac{d^2}{dt^2} K_0(t) + \frac{1}{t} \frac{d}{dt} K_0(t) - K_0(t) = 0, \quad t>0
  \end{equation}
  and it is the unique solution to \eqref{eq:BesselODE.0} with the
  following boundary behaviour, see e.g.\
  \cite[\href{https://dlmf.nist.gov/10.25.E3}{10.25.3}]{NIST:DLMF} and
  \cite[\href{https://dlmf.nist.gov/10.30.E3}{10.30.3}]{NIST:DLMF},
  \begin{align}
    \label{eq:BesselK0.bdry}
    K_0(t) \sim \sqrt{\frac{\pi}{2t}} \exp(-t) \text{ for }t \to\infty, \qquad
    K_0(0) \sim \log(1/t) \text{ for }t \to 0 .
  \end{align}
  %% {Behaviour for large $\abs{\mathbf{y}}$ (and small $\mu$)}
  %We have $K_0(z) \sim (\pi/2z )^{1/2} \exp(-z)$ for $z \to \infty$
  %(e.g.\ Abramowitz \& Stegun, 9.7.2).
\end{remark}

\begin{remark}[Brownian motion viewpoint on Mal\'ecot's formula]
  As observed e.g.\ in \cite[Appendix]{BartonDepaulisEtheridge:2002}, 
  the Wright-Mal\'ecot formula \eqref{eq:MF} can alternatively be viewed
  as an assumption on the behaviour of the spatial embedding of two
  sampled ancestral lineages, at least while their spatial separation
  is larger than $\kappa$.  Namely, let us assume that two ancestral
  lineages, which where sampled with an initial separation $x$, cannot
  merge while separated by more than $\kappa$ and until that time the
  difference of their spatial embeddings behaves like a $2d$-Brownian
  motion $(B_t)_{t \ge 0}$ with variance $2\sigma^2$.
  (Here, we replaced discrete generations by continuous time, and we 
  implicitly assume that the mutation rate $\mu>0$ is small so that
  the replacement $(1-\mu)^{2 \tau} \approx e^{-2\mu\tau}$ is justified.)
  
  Put $\tau \coloneqq \inf\{t \ge 0 : \norm{B_t} \le \kappa\}$ and set
  \begin{equation}
    f_\mu(x) = \bE_x[\exp(-2\mu \tau) \phi(\kappa)], \quad x \in \R^2
  \end{equation}
  (where $\bE_x$ refers to the intial condition $B_0=x$). It is well known that 
  $f_\mu$ is the unique solution of
  \begin{equation}
    \begin{split}
      % \frac12
      \sigma^2 \Delta f_\mu(x) - 2\mu f_\mu(x) = 0, & \quad \norm{x}>\kappa, \\
      f_\mu(x) = \phi(\kappa), & \quad \norm{x} \le \kappa.
    \end{split}
  \end{equation}
  % (Sketch: $f_\lambda(X) \approx \bE_x[e^{-\lambda h} f_\lambda(B_h)] \approx (1-\lambda h) \bE_x[f_\lambda(B_h)]
  % = (1-\lambda h) \bE_x\big[ f_\lambda(x) + \int_0^h \frac12 \Delta f_\lambda(B_s) \, ds \big]
  % \approx (1-\lambda h)\big( f_\lambda(x) + h \frac12 \Delta f_\lambda(x)\big)
  % = f_\lambda(x) + h \big(-\lambda f_\lambda(x) + \frac12 \Delta f_\lambda(x)\big) + O(h^2)$,
  % thus $-\lambda f_\lambda(x) + \frac12 \Delta f_\lambda(x) = 0$.)
  By radial symmetry, $f_\mu(x) = g_\mu(\norm{x})$, where 
  $g_\mu$ solves (using the Laplace operator in polar coordinates)
  \begin{equation}
    \label{eq:g_mu}
    g_\mu''(r) + \frac1{r}g_\mu'(r) - \frac{2\mu}{\sigma^2} g_\mu(r) = 0, \quad r \ge \kappa
    \qquad \text{with} \quad g_\mu(\kappa)=\phi(\kappa), \; g_\mu(\infty)=0.
  \end{equation}
  Note that \eqref{eq:g_mu} coincides with the modified Bessel ODE \eqref{eq:BesselODE.0}
  up to a re-scaling (indeed, $\widetilde{g}(t) \coloneqq g_\mu(\sigma t/\sqrt{2\mu})$ solves \eqref{eq:BesselODE.0}
  for $\sqrt{2\mu}\kappa/\sigma < t < \infty)$), hence
  \[
    g_\mu(r) = \frac{\phi(\kappa)}{K_0\big(\sqrt{2 \mu} \hspace{0.1em} \kappa /\sigma \big)}
    K_0\big(\sqrt{2 \mu} \hspace{0.1em} r /\sigma \big)
    \quad \text{for } \; r \ge \kappa.
  \]

  The assumption on the behaviour of sampled ancestral lineages discussed above
  together with the strong Markov property of Brownian motion (and our assumption of
  mutations according to the infinite alleles model) implies then for $\norm{x} \ge \kappa$
  \begin{equation}
    \phi(x) = f_\mu(x) = g_\mu(\norm{x}) =
    \frac{\phi(\kappa)}{K_0\big(\sqrt{2 \mu} \hspace{0.1em} \kappa /\sigma \big)}
    K_0\big(\sqrt{2 \mu} \hspace{0.1em} \norm{x} /\sigma \big),
  \end{equation}
  which agrees with the right-hand side of \eqref{eq:MF}.

   But note that this argument shows the form of $\phi(x)$ involving the Bessel function, but in itself does not identify the prefactor. One can of course insert the value of $\phi(\kappa)$ from the right-hand side of \eqref{eq:MF} and then literally recover \eqref{eq:MF}, which seems however a bit circular.
   This is an instance of the fact that we are treating here (as well as later in the proofs in Section~\ref{sect:main.proof}) what happens when the two lineages are close as something like a ``black box'' where we have no explicit control.

\end{remark}

\begin{remark}[Hitting times for two-dimensional random walk]
  Let $(S_k)_{k \in \N_0}$ be an irreducible, centered 2-dimensional random walk with finite second moments
  and let $\tau_0 := \inf\{ k : S_k = 0\}$ be the hitting time of the origin. For $x(n) \in \bZ^2$ with $\norm{x(n)}_2 \to \infty$ we have 
  \begin{align}
    \label{eq:ErdosTaylor}
    \bP\big( \tau_0 > n \,\big|\, S_0=x(n)\big) \sim \frac{2 \log(||x||_2)}{\log n} \wedge 1 \qquad \text{as } n \to\infty.
  \end{align}
  This was first proved by Erd\H{o}s and Taylor \cite{ErdosTaylor:1960} for symmetric simple random walk on $\bZ^2$ (see \cite[(2.16)]  {ErdosTaylor:1960} and read $\rho = ||x(n)||_2$, assuming $\norm{x(n)}_2 \le n^{1/3}$, cf.\ \cite[(2.13)]{ErdosTaylor:1960}) and later extended in \cite[Thm.~3]{Sawyer:1977a}. 
  These proofs are computational and do make use of the explicit structure of a random walk as a sum of i.i.d.\ steps, which allows fairly explicit computation of generating functions as well as renewal decompositions according to returns to the origin and local CLT estimates. We refer to \cite[p.~354]{CoxGriffeath:1986} for a nice heuristic explanation of the idea behind them.
  \smallskip
  
  Recalling \eqref{eq:limit_coalescence_time} from Theorem~\ref{thm:Pct} we see that our main result can be viewed as establishing the Erd\H{o}s-Taylor asymptotics\ \eqref{eq:ErdosTaylor} for the difference $(S_k)_k \coloneqq (X_k-X'_k)_k$ of the two ancestral lines. Intuitively, this holds because even though the law of $X-X'$ is complicated (it is in particular not a random walk, unlike the situation in the stepping stone model, see Remark~\ref{rem:steppingstone} below), it is close to a random walk when the separation $X_k-X'_k$ is large (see the discussion in Section~\ref{sect:ResultsfromBDS24} and especially \eqref{eq:TVdistance-joint-ind-1step in prop}). In this sense, our study belongs to the circle of results which are sometimes called ``Lamperti problems'' in honour of John Lamperti's work in the 1960ies on locally perturbed random walks, see e.g.\ \cite{MPW17} for background and references.
    
 Our proof of Theorem~\ref{thm:Pct} uses the simple (and robust) idea that for $\gamma > 1$ and $N \gg M \gg 1$
 \begin{align*}
   \bP\big( \tau_0 > N^{2\gamma} \,\big|\, \norm{S_0}_2=N\big) 
   & \approx \bP\big( \norm{S}_2 \text{ hits $N^\gamma$ before dropping below $M$} \,\big|\, \norm{S_0}_2=N\big) \\
   & \approx \frac{\log(N) - \log(M)}{\log(N^\gamma)-\log(M)} \approx \frac{1}{\gamma}
 \end{align*}
 because $\bZ^2 \ni x \mapsto \log(\norm{x}_2)$ is ``almost harmonic'' for $S = X-X'$ (it is of course literally harmonic for 2d Brownian motion). See also Remark~\ref{rem:separate} below.
\end{remark}

\begin{remark}[Known problems with Mal\'ecot's formula]
  \label{rem:MFproblems}
  As mentioned above, the derivation of \eqref{eq:MF} in
  \cite{Malecot:1948, Malecot:1969} is based on the recursion
  \eqref{eq:MFrecursion} via `backwards in time analysis', however,
  there is no rigorous underlying forwards-in-time stochastic
  population model in the derivation. It is well known that in two
  dimensions the critical branching random walk dies out locally and
  builds large clumps in the regions in which it does survive. Similarly on compact
  spaces, such as a torus, the critical branching random walk dies out almost surely
  and forms arbitrarily dense clumps if conditioned to survive.
  J.~Felsenstein called this phenomenon the `pain in the torus' in
  \cite{Felsenstein1975}. It is also well known that critical
  branching Brownian motion dies out in $d=2$, see e.g.\
  \cite{Kallenberg:1977}, \cite{BramsonCoxGreven1993}. Another problem
  is that there is no obvious way (nor in fact a consistent way) of
  extending the recursions backwards in time to larger sample sizes.
  \smallskip 

  \noindent
  There are some remedies to overcome 
  these issues that were considered in the literature.
  \begin{itemize}
  \item[(a)] A class of models known as stepping stone models is very
    popular. Here, the space is discretised and constant local
    population sizes are enforced. One can think in these models that
    in the forwards in time evolution of type configurations
    individuals in the new generation choose their parents at random
    from some neighbourhood in the previous generation and adopt their
    type (at this stage it is possible to introduce selection and mutation). 
    Here the ancestral lineages perform random walks and
    accordingly the ancestries of samples of individuals perform
    coalescing random walks with a coalescence delay depending on the
    (constant) local population size.

    Stepping stone models were introduced by Kimura and Weiss in
    \cite{KimuraWeiss1964} and different flavours of these models as
    well as different questions were studied in
    \cite{WeissKimura:1965}, \cite{Maruyama1970}, \cite{Malecot1975},
    \citep{Sawyer:1976}, \cite{Shiga1981}, \cite{Shiga1988},
    \cite{Wilkinson-Herbots:1998, Wilkinson-Herbots:2003} and many
    others. For an overview see for example Chapter 6 in
    \cite{Etheridge:2011}.

    These models lead to elegant sampling formulas via duality with
    coalescing random walks (see also Remark~\ref{rem:steppingstone}
    below) and it has been observed that the resulting probability of identity as a function of sampling distance (see \eqref{eq:1} below) agrees already for rather moderate separation $x$ very well with $\phi(x)$ from \eqref{eq:MF} with suitably adjusted parameters, see e.g.\  \cite[Fig.~1]{BartonDepaulisEtheridge:2002}.
    On the other hand, the deterministic size restriction seems artificial and
    in particular not suitable for `ecological' spatial stochastic
    models.

  \item[(b)] In order to overcome the restriction to constant local
    population size one can consider branching random walks or related
    processes with local regulation. Here the offspring distribution
    is supercritical in sparsely populated regions and subcritical
    when there are many neighbours. The models from this class are
    natural extensions of stepping stone models, of branching random
    walks as well as of the contact process.

    There is a large body of literature in which such models are
    studied. For instance Bolker and Pacala \cite{BolkerPacala:97,
      BolkerPacala:99}, Law and Dieckmann  \cite{LawDieckmann:02},
    Murrell and Law \cite{MurrellLaw2003} study such models based on simulations and non-rigorous moment closure approximations.
    Mathematically rigorous analyses were carried out by Etheridge 
    \cite{Etheridge:2004}, Fournier and M\'el\'eard 
    \cite{FournierMeleard:2004}, Hutzenthaler and Wakolbinger 
    \cite{HutzenthalerWakolbinger:07}, Blath, Etheridge and Meredith 
    \cite{BlathEtheridgeMeredith:2007}, Birkner and Depperschmidt 
    \cite{BirknerDepperschmidt2007}, Pardoux and Wakolbinger 
    \cite{PardouxWakolbinger2011}, Finkel\-shtein, Kondratiev and Kutoviy \cite{FinkelshteinKondratievKutoviy2012}, Le, Pardoux and Wakolbinger
    \cite{LePardouxWakolbinger:2013}, Greven, Sturm, Winter and Z\"ahle
    \cite{greven2015multitype}, Maillard and Penington
    \cite{MaillardPenington2022}, Etheridge, Kurtz, Letter, Ralph and Tsui \cite{Etheridge+al:2024} and others.
    Many aspects of models with local regulation have been studied, see e.g.\ \cite{BirknerGantert21} for a partial overview and further discussion.
    However, there is little hope for explicit sampling formulas.
    
    \item[(c)] In a series of papers starting with the work by A.~M.~Etheridge \cite{Etheridge2008} and \cite{BartonEtheridgeVeber:2010} together with N.~H.~Barton and A.~V\'eber, the spatial-$\Lambda$-Fleming Viot process has been introduced and studied. Several properties of the process as well as extensions were considered in \cite{VeberWakolbinger:2015,EtheridgeFreemanStraulino2017,EtheridgeKurtz2019,EtheridgeVeberYu2020}. 
      The process solves the problems described at the beginning of this remark and 
      in principle elegant sampling formulas are available. In a suitable scaling limit, there is even an explicit analogue of the Wright-Mal\'ecot formula \eqref{eq:MF}, see \cite[Theorem~2.4]{Forien2022}.
      On the other hand, the derivation of the spatial-$\Lambda$-Fleming Viot process implicitly takes a large (local) population density limit and thus does not incorporate the possibility of local size fluctuations. 

  \item[(d)] One can consider the Fleming-Viot process (see e.g.\
    \cite{Etheridge:SuperprocessesBook}) or one of its `multiple
    merger' generalisations, a $\Lambda$- or $\Xi$-Fleming-Viot
    process (see e.g.\ \cite{DonnellyKurtz1999, BertoinLeGall2003,
      Birkner.et.al2009}) with underlying Brownian motion on a
    compact, two-dimensional continuous space like a suitable bounded
    domain in $\bR^2$ with periodic or reflecting boundary conditions.
    These processes have a fixed total mass, hence there is no problem
    with extinction nor with clumping, and observed (only) locally,
    the population size does fluctuate very much like a critical
    branching process (albeit not in a way that decorrelates with
    distance because of the total mass constraint).  They also have
    (suitably interpreted) sampling consistency for arbitrary sample
    sizes, see \cite{Koepernik2024}, and the ancestry of samples
    together with their spatial embeddings has an elegant
    description via so-called Brownian spatial coalescents; there is
    also an analogue of the Wright-Mal\'ecot formula in this context
    (\cite[Example~1.9]{Koepernik2024}).

    As observed in \cite[Section~1.7.3]{Koepernik2024}, for $d \ge 3$,
    the restriction to a compact (geographical) space can be overcome by
    considering an analogous programme based on the stationary version
    of superbrownian motion (which exists on $\bR^d$ for any $d\ge 3$,
    see e.g.\ \cite{Etheridge:SuperprocessesBook}).
  \end{itemize}

  As an aside, let us remark that regulating the branching rate
  based on the local population density while keeping the offspring
  law critical does not resolve the problem, as conjectured by
  Alison Etheridge in the early 2000s and later proved in
  \cite{BirknerSun:2019}. \smallskip

  Let us also note that in the context of classical interacting
  particle systems (e.g.\ summarised in \cite{Liggett:1999}),
  stepping stone models correspond to the voter model whereas the
  locally regulated models correspond to the contact process.
\end{remark}

\begin{remark}[Stepping stone models and relatives and identity by descent]
  \label{rem:steppingstone}
  Colonies of \emph{fixed} local size $N$ are arranged in a
  geographical space, say $\Z^d$.
  For each individual in colony $x$, with probability $p(x,y) =
  p(y-x)$ assign a random parent in previous generation from colony
  $y$. In these models the demographic structure is trivial,
  nevertheless these models are paradigm models for evolution of
  \emph{type distribution} in space.

  Ancestral lineages in such a model are coalescing random walks:
  Sample one individual from colony $x$ and one from colony $y$. The
  spatial positions of the ancestral lines are random walks with
  (delayed) coalescence. While not yet merged, each takes an
  independent step according to the random walk transition kernel $p$,
  every time they are in the same colony, the two lines merge with
  probability $1/N$.  Note that here, the recursive decomposition
  behind \eqref{eq:MFrecursion} becomes exact.

  Assume that during the reproduction each offspring mutates with
  probability $u$ (to a completely new type), let
  \[
    \psi(x,y) := \parbox[t]{0.7\textwidth}{probability in
      equilibrium that two individuals\\ randomly drawn from
      colonies $x$ and $y$ have same type}
  \]
  Assuming that $p$ is symmetric we have
  \begin{align}
    \label{eq:1}
    \psi(x,y) = \frac{1-\psi(0,0)}{N} \sum_{k=1}^\infty (1-u)^{2k} p_{2k}(x,y)
    = \frac{G_u(x,y)}{N+G_u(0,0)}
  \end{align}
  where $p_k$ is the $k$-step transition kernel and
    \begin{align}
    \label{eq:9}
    G_u(x,y) = \sum_{k=1}^\infty (1-u)^{2k} p_{2k}(x,y).
  \end{align}
  The first identity in \eqref{eq:1} is computed e.g.\ in Theorem~5.4
  in \cite{Durrett:PM4DNASE2ndEd} where one has to replace $2N$ by $N$
  in our case. The second identity in \eqref{eq:1} follows by using
  the first, solving for $\psi(0,0)$ and then substituting into the
  first identity.

  For the behaviour of $\psi(x,y)$ for $\norm{x-y} \to \infty$ we
  assume that $p$ is irreducible and can be written for some $\nu \in
  (0,1]$ as
  \begin{align}
    \label{eq:10}
    p(x,y) = (1-\nu) \delta_{x,y} + \nu q(x,y),
  \end{align}
  where $q$ is a symmetric translation invariant stochastic kernel on
  $\Z^2$ with $q(0,0)=0$, finite third moments and covariance matrix given by $\sigma^2$ times the two dimensional unit matrix. By Theorem~5.7 in
  \cite{Durrett:PM4DNASE2ndEd}, setting $\ell = (\nu\sigma^2/(2
  u))^{1/2}$, we have
  \begin{align}
    \label{eq:11}
     \psi(x,y) \sim \frac{1}{2\pi N + \log \ell}
    \Bigl(K_0(\norm{x-y}/\ell) - K_0(\norm{x-y})\Bigr),
  \end{align}
  where $\sim$ denotes asymptotic equivalence as $u \to 0$, and $K_0$
  is the modified Bessel function from \eqref{eq:BesselK0.bdry}.
  
  For an incomplete list of literature on different flavours of the stepping stone models we refer to Remark~\ref{rem:MFproblems}(a). Let us note that \cite{BartonDepaulisEtheridge:2002} observe a very good agreement with same quantity for the (discrete space) stepping stone
  model, cf.\ \cite[Fig.~1]{BartonDepaulisEtheridge:2002}. Also ``explicit'' results for the stepping stone model on the two-dimensional grid are obtained e.g.\ in \cite[Sect.~4.5]{Wilkinson-Herbots:1998}, where it is in fact also discussed that Mal\'ecots
  formula gives a good approximation of the result obtained there.
\end{remark}

\subsection{Outlook}
\label{sect:outlook}

As \cite[Section~7]{Etheridge2008} observes, ``it is widely believed that if one views populations over sufficiently large spatial and temporal scales, then there should be some averaging effect which would allow one to use classical population genetic models with constant population density but with effective parameters replacing the real population parameters''. We share this belief and hope that our results help to corroborate it. Still, many questions concerning ancestral lineages in locally regulated models remain open:
\begin{itemize}
\item As they stand, Theorem~\ref{thm:Pct} and Corollary~\ref{cor:Malecot} are `conceptual' rather than practical. They describe a mathematical limit without controlling how large $N$ needs to be and we have no explicit information about required conditions on the competition parameters. In addition, we identify the decay behaviour but our limit result does not capture a prefactor as in \eqref{eq:MF}.
Concerning the (very pertinent) question how to get numbers out of it, we would presently have to resort Monte Carlo estimates via computer simulations.

It would also be interesting to describe the asymptotic variance of an ancestral lineage from LBRW more explicitly. The regeneration construction from \cite{BirknerCernyDepperschmidt2016, BirknerDepperschmidtSchlueter2024} in principle gives a possible answer (namely, the variance of inter-regeneration increments divided by the mean waiting time between regenerations) but there seems no easy way to actually compute these.
\item On the other hand, the `abstract' regeneration construction from \cite{BirknerCernyDepperschmidt2016, BirknerDepperschmidtSchlueter2024} is in principle very flexible and allows to cover more general spatial population models with local regulation than just LBRW. We believe that any such model which has a `supercritical phase' where it can be compared to the discrete time contact process could be accommodated in this framework. It would be also very interesting to see whether continuous-time models could be treated in this way but this will require new ideas, maybe as in \cite{MaillardPenington2022}, because in continuous time there is no a priori bound on the `propagation of information'.

\item We focused here on the case $d=2$ but the cases $d=1$ and $d \ge 3$ could be treated as well. We believe by analogy with ordinary random walk that in $d=1$, coalescence times will have square root tails and have a positive chance to be $+\infty$ in $d\ge 3$ (in fact, the latter follows from the arguments in \cite{BirknerDepperschmidtSchlueter2024}). This in particular implies that for LBRW in $d \ge 3$ with neutral genetic types there exist multi-type equilibria.

\item A natural question would be to consider larger sample sizes than just two. For a related simpler model, the discrete time contact process in $d=1$, this was studied in \cite{BirknerGantertSteiber2019} and a Brownian web limit was established. It is conceivable that the same will hold for LBRW in $d=1$. It is also conceivable that in $d=2$, (neutral) multitype LBRW will exhibit diffusive clustering analogous to the classical voter model \cite{CoxGriffeath:1986}.

\item Last but not least it would be highly interesting to consider types with selective differences in locally regulated models and their effect on ancestral lineages.
\end{itemize}

\subsection{Outline}
We recall important results and tools from
\cite{BirknerDepperschmidtSchlueter2024} (and from
\cite{BirknerCernyDepperschmidt2016}) in
Section~\ref{sect:ResultsfromBDS24}; the proof of
Theorem~\ref{thm:Pct} is given in Section~\ref{sect:main.proof}, based
on a series of intermediate results that build on the tools from
Section~\ref{sect:ResultsfromBDS24}, with some proof details and
longer calculations relegated to Appendices~\ref{sect:proof of 1.6}--
\ref{sect:Coalescing probabilities}.

\section{Summary of relevant results from \cite{BirknerDepperschmidtSchlueter2024}}
\label{sect:ResultsfromBDS24}

In this section, we recall concepts, tools and results from \cite{BirknerDepperschmidtSchlueter2024}
(and from \cite{BirknerDepperschmidt2007}, \cite{BirknerCernyDepperschmidt2016}) which will be needed in the proofs in Section~\ref{sect:main.proof}.

  \subsection{Ancestral lineages in LBRW}
  \label{subsect:Ancestral lineages in LBRW}
  As discussed in Section~\ref{sect:Ancestral lineages}, we enrich the logistic branching random walks from \eqref{eq:law of eta in log branching} from Section~\ref{sect:log branching random walks} with genealogical information
(cf Chapter~4 in \cite{Depperschmidt08} or Section~4 in \cite{BirknerCernyDepperschmidt2016}). Conditioned on the space-time configuration $\eta$, the spatial embedding $X$ of an ancestral lineage from Section~\ref{sect:Ancestral lineages} is a time-inhomogeneous Markov chain with transition probabilities 
\begin{align}
  \label{eq:defX-real}
  \Pr(X_{k+1}=y \, | \, X_k=x,\eta) = 
  \frac{p_{yx}f(y;\eta_{-k-1})}{\sum_z p_{zx}f(z;\eta_{-k-1})} 
  \eqqcolon
  p_{ \eta}(k; x,y), \quad x,y \in
  \Z^d, \, k \in \Z_+
\end{align}
(with some arbitrary convention if the denominator is $0$).
The joint transition dynamics of the pair $(X,X')$ given $\eta$ is the product of terms as in \eqref{eq:defX-real} when $X_k \neq X'_k$ and contains the possibility of coalescence when they jump to the same site. See Appendix~\ref{sect:Coalescing probabilities} for details. Note that $p_{ \eta}(k; x,y)$ is close to $p_{yx}$ when $\eta_{-k-1}$ has small relative fluctuations in a neighbourhood of $x$.

\subsection{Coupling properties of $\eta$}

First we start with $\eta$, the process describing the evolution of the population. Recall that $\eta_n(x)$ is the number of individuals at position $x\in\bZ^2$ in generation $n$. For $m\in(1,3)$, $0<\lambda_0\ll 1$ and $\lambda_z\ll \lambda_0$ for $z\neq 0$, by Theorem~\ref{thm:eta_survival}, $\eta$ survives with positive probability and, conditioned on survival, converges in distribution to $\eta^{\rm stat}$. 
The key idea behind this result is to construct $(\eta_n)$ as in \eqref{eq:law of eta in log branching} as a function of a space-time `driving noise' which takes the form of a space-time system of independent Poisson processes, see \cite[Section~4.1]{BirknerCernyDepperschmidt2016}.

The corresponding deterministic model is a dynamical system $\zeta\coloneqq (\zeta_n)_n$ (also called coupled map lattices) on $[0,\infty)^{\bZ^2}$ defined by
\begin{equation}
        \label{eq:defn_zeta}
        \zeta_{n+1}(x) = \sum_{y\in\bZ^2} p_{x-y} \zeta_n(y) \Big( m
        - {\textstyle \sum_{z\in\bZ^2} \lambda_{z-y} \zeta_n(z) } \Big)^+ = \sum_{y\in\bZ^2} p_{x-y}f(y;\zeta_n),
\end{equation}
where $f$ here is the function introduced in \eqref{def:f}. Think of $\zeta_n(x)$ as the expected number of individuals at site $x$ in generation $n$. By \cite{BirknerDepperschmidt2007}, under the assumptions in Theorem~\ref{thm:eta_survival}, $(\zeta_n)$ has a unique non-trivial fixed point, i.e.\ $\zeta_n(x)$ converges to $(m-1)/\sum_z \lambda_z$ for all $x\in \bZ^2$. This is the crucial property which makes the coupling \eqref{eq:contraction}--\eqref{eq:propagation.coupling} below possible.

We interpret the ancestral lines as random walks in a dynamic random environment generated by the time reversal of $\eta^{\rm stat}$. Our approach to handle this environment is to link the model with supercritical oriented percolation using a coarse-graining technique. That is, for $L_s,L_t\in \bN$ we divide the $\bZ^2\times \bZ$ into space-time boxes whose `bottom parts' are centred at points on the coarse-grained grid $L_s\bZ^2\times L_t\bZ$ and define 
\begin{equation*}
    \mathsf{block}(\tilde{x},\tilde{n}) \coloneqq \big\{ (y,k)
  \in \Z^d \times \Z \, : \,
  \norm{y-L_{\mathrm{s}} \tilde{x}} \le
  L_{\mathrm{s}}, \tilde{n} L_{\mathrm{t}} < k \le
  (\tilde{n}+1) L_{\mathrm{t}}\big\}.
\end{equation*}
In \cite{BirknerCernyDepperschmidt2016} it was shown that $\eta$ satisfies a certain coupling property for the range of parameters we consider in Theorem~\ref{thm:Pct}. To be more precise, there exists a finite set of `good' local
configurations $G_{\eta} \subset \Z_+^{B_{2 L_{\mathrm{s}}}(0)}$ such that for any $(\tilde{x}, \tilde{n}) \in \Z^2 \times \Z$ and any
    configurations
    $\eta_{\tilde{n} L_{\mathrm{t}}}, \eta'_{\tilde{n} L_{\mathrm{t}}} \in
    \Z_+^{\Z^2}$ at time $\tilde{n} L_{\mathrm{t}}$,
    \begin{align}
      \label{eq:contraction}
      \begin{split}
        & \restr{\eta_{\tilde{n} L_{\mathrm{t}}}}{B_{2 L_{\mathrm{s}}}
          (L_{\mathrm{s}} \tilde{x})}, \, \restr{\eta'_{\tilde{n}
            L_{\mathrm{t}}}}{B_{2 L_{\mathrm{s}}} (L_{\mathrm{s}}
          \tilde{x})} \in G_{\eta} 
        \\[+0.5ex]
        & \Rightarrow \;\; \eta_{(\tilde{n}+1) L_{\mathrm{t}}}(y) =
        \eta'_{(\tilde{n}+1) L_{\mathrm{t}}}(y) \quad \text{for all $y$
          with} \; \norm{y-L_{\mathrm{s}}
          \tilde{x}} \le 3 L_{\mathrm{s}}  \\
        & \qquad \text{and} \quad \restr{\eta_{(\tilde{n}+1)
            L_{\mathrm{t}}}}{B_{2 L_{\mathrm{s}}} (L_{\mathrm{s}}
          (\tilde{x}+\wt{e}))} \in G_{\eta} \; \text{for all $\wt{e}$
          with} \; \norm{\wt{e}} \le 1,
      \end{split}
      \\ \intertext{whenever the driving noise is `good' and}
      \label{eq:propagation.coupling}
        & \restr{\eta_{\tilde{n} L_{\mathrm{t}}}}{B_{2 L_{\mathrm{s}}}(L_{\mathrm{s}} \tilde{x})}
          =  \restr{\eta'_{\tilde{n} L_{\mathrm{t}}}}{B_{2 L_{\mathrm{s}}}(L_{\mathrm{s}} \tilde{x})}
          \quad \Rightarrow \quad
          \eta_k(y) = \eta'_k(y) \;\; \text{for all} \; (y,k) \in
          \mathsf{block}(\tilde{x},\tilde{n}),
    \end{align}
where $\eta=(\eta_n)$ and $\eta'=(\eta'_n)$ are given by
\eqref{eq:law of eta in log branching} with the same driving noise but possibly different
initial conditions. In essence we are in a setting where `good' local configurations propagate from one box-level to the next given the local driving noise is `good'. Therefore, if one can show that good local randomness has high probability (typically close to 1), one can couple this propagation of `good' configurations to supercritical oriented percolation. Or phrased in a less mathematical way as in \cite[p.~1780]{BirknerDepperschmidt2007} ``life plus good randomness leads to more life, so show that bad randomness has small probability''.

\subsection{A joint regeneration construction}
\label{sect:joint regeneration construction}
We observe the random walks only along these boxes (we typically think of $L_t>L_s\gg \max\{R_p,R_\lambda\}$). More specifically, let $X$ be an ancestral lineage evolving in an environment generated by $\eta^{\rm stat}$, then we define the corresponding coarse-grained random walk $\wt{X}$ by
\begin{align*}
        \wt{X}_{n} \coloneqq \wt{\pi}(X_{nL_t}),
\end{align*}
where $\wt{\pi}\colon \bZ^2\to \bZ^2$ is the coarse-graining function
\begin{equation*}
        \wt{\pi}(x) = \wt{\pi}(x_1,x_2)=(\wt{x}_1,\wt{x}_2)\coloneqq \Big( \Big\lceil \frac{x_1}{L_s}-\frac{1}{2}  \Big\rceil,\Big\lceil \frac{x_2}{L_s}-\frac{1}{2} \Big\rceil \Big).
\end{equation*}
Thus, $\wt{X}$ observes $X$ only along the space-time boxes. In \cite{BirknerDepperschmidtSchlueter2024} we studied the behaviour of a pair of coarse-grained random walks $(\wt{X},\wt{X}')$ evolving in the same environment to prove a quenched CLT for $X$. The aim of this paper is to prove Theorem~\ref{thm:Pct} in which we study the coalescing time of two ancestral lineages. It is obvious that the two lineages will have to enter the same space-time box before they are able to coalesce. Therefore we approach this by first studying when the coarse-grained versions $(\wt{X},\wt{X}')$ of these two ancestral lineages meet and estimate the error terms to then obtain the statement for $(X,X')$. In the following we recall useful results for the coarse-grained random walks we obtained in \cite{BirknerDepperschmidtSchlueter2024} in the process of proving a quenched CLT for $(X,X')$. Note that, although we call $\wt{X}$ a coarse-grained random walk, it is in fact not exactly a random walk. But we invite the reader to think of it since this guides the proofs (treating the actual dynamics mostly leads to additional technical difficulties).

To study the behaviour of the two coarse-grained ancestral lineages we define a regeneration construction in \cite{BirknerDepperschmidtSchlueter2024}; see Construction~\regenerationConstr {} therein. The purpose of these regeneration times is to split the random walk into increments that evaluate independent parts of the environment. We construct a double cone that encompasses the paths of both random walks and the part of environment $\eta$ both walkers see during an increment. Furthermore, we allow regenerations only at times where we can ensure that the values of $\eta$ that were evaluated up to the regeneration have no influence on the values of $\eta$ in the future of the random walks. This construction isolates the values of $\eta$ inside the cone from the values outside, i.e.\ the values of the environment inside the cone are independent of the values outside. For a visual representation of the cone construction see Figure~\ref{fig:dcones}. Note that this isolation essentially stems from the fact that the construction is made in such a way that every path on the coarse-grained level has to hit a box where the local configuration is `good' and the local driving noise is `good' and thus the values of $\eta$ inside this box are determined only by that local noise. The ancestral lineages and all information about the environment they gather are contained in the inner cone. Say we know the value of $\eta$ somewhere outside the cone shell (in Figure~\ref{fig:dcones} illustrated as the red dot); by the dynamics of the environment this information propagates along a path that has to cross the cone shell (grey in Figure~\ref{fig:dcones}). By construction we only regenerate when the cone shell isolates the inner cone (green) from the environment outside and thus the cone shell ``shields'' the environment in the inner cone from the information along the red path. For a more fleshed out and mathematical discussion of the construction we refer to \cite[Construction~\regenerationConstr]{BirknerDepperschmidtSchlueter2024}.

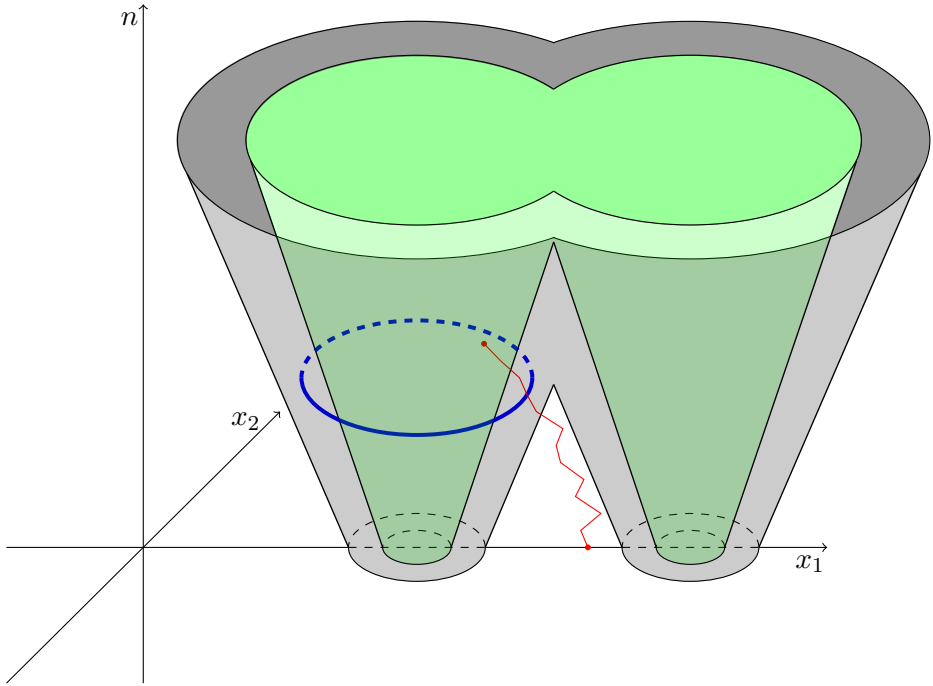
\begin{figure}
        \centering
        \begin{tikzpicture}[yscale=0.9,xscale=0.9]
                \definecolor{drawColor}{RGB}{0,0,0}

                % Achsen

                \draw[->] (-6,-2) -- (-2,2);
                \draw (-6,0) -- (-1,0);
                \draw[dashed] (-1,0) -- (1,0);
                \draw (1,0) -- (3,0);
                \draw[dashed] (3,0) -- (5,0);
                \draw[->] (5,0) -- (6,0);
                \draw[->] (-4,-2) -- (-4,8);

                % Achsenbeschriftungen

                \node[text=drawColor,anchor=base,inner sep=0pt, outer sep=0pt, scale=  1.00] at (-4.2,7.7) {$n$};
                \node[text=drawColor,anchor=base,inner sep=0pt, outer sep=0pt, scale=  1.00] at (-2.5,1.8) {$x_2$};
                \node[text=drawColor,anchor=base,inner sep=0pt, outer sep=0pt, scale=  1.00] at (5.75,-.3) {$x_1$};

                % linker äußerer Kegel
                % \draw (0,0) ellipse (1cm and .5cm);
                \draw[dashed] (1,0) arc (0:180: 1cm and .5cm);
                \draw (-1,0) arc (180:360: 1cm and .5cm);
                % \draw (0,6) ellipse (3cm and 1.5cm);
                \draw (-3.5,6) arc (180:55: 3.5cm and 1.75cm);
                \draw (-3.5,6) arc (180:305: 3.5cm and 1.75cm);
                \path[draw=drawColor,line width= 0.5pt,line join=round] (-1,0) -- (-3.35882,5.50116);
                \path[draw=drawColor,line width= 0.5pt,line join=round] (1,0) -- (2,2.4);

                % linker innerer Kegel
                % \draw (0,0) ellipse (0.5cm and .25cm);
                \draw[dashed] (0.5,0) arc (0:180: 0.5cm and .25cm);
                \draw (-0.5,0) arc (180:360: 0.5cm and .25cm);
                % \draw (0,6) ellipse (2.5cm and 1.25cm);
                \draw (-2.5,6) arc (180:37: 2.5cm and 1.25cm);
                \draw (-2.5,6) arc (180:323: 2.5cm and 1.25cm);
                \path[draw=drawColor,line width= 0.5pt,line join=round] (-0.5,0) -- (-2.428649,5.6945946);
                \path[draw=drawColor,line width= 0.5pt,line join=round] (0.5,0) -- (2,4.5);

                % rechter äußerer Kegel
                % \draw (4,0) ellipse (1cm and .5cm);
                \draw[dashed] (5,0) arc (0:180: 1cm and .5cm);
                \draw (3,0) arc (180:360: 1cm and .5cm);
                % \draw (4,6) ellipse (3cm and 1.5cm);
                \draw (7.5,6) arc (0:125: 3.5cm and 1.75cm);
                \draw (7.5,6) arc (360:235: 3.5cm and 1.75cm);
                \path[draw=drawColor,line width= 0.5pt,line join=round] (3,0) -- (2,2.4);
                \path[draw=drawColor,line width= 0.5pt,line join=round] (5,0) -- (7.35882,5.50116);

                % rechter innerer Kegel
                % \draw (4,0) ellipse (0.5cm and .25cm);
                \draw[dashed] (4.5,0) arc (0:180: 0.5cm and .25cm);
                \draw (3.5,0) arc (180:360: 0.5cm and .25cm);
                % \draw (4,6) ellipse (2.5cm and 1.25cm);
                \draw (6.5,6) arc (0:143: 2.5cm and 1.25cm);
                \draw (6.5,6) arc (360:217: 2.5cm and 1.25cm);
                \path[draw=drawColor,line width= 0.5pt,line join=round] (3.5,0) -- (2,4.5);
                \path[draw=drawColor,line width= 0.5pt,line join=round] (4.5,0) -- (6.428649,5.6945946);

                % Mittelschlauch

                \draw[blue,line width= 1.5pt] (-1.6875,2.5) arc (180:360:1.6875cm and 0.84375cm);
                \draw[dashed,line width= 1.5pt,blue] (1.6875,2.5) arc (0:180:1.6875cm and 0.84375cm);

                % Pfad mit Start- und Endpunkt
                \draw[red] (2.5,0) -- (2.3959730,0.25) -- (2.6913578,0.5) -- (2.3182013,0.75) -- (2.4424640,1) -- (2.1008534,1.25) -- (2.0371980,1.5) -- (2.1338517,1.75) -- (1.7472884,2) -- (1.6095703,2.25) -- (1.5,2.5) -- (1.2227309,2.75) -- (0.9823701,3);

                \filldraw[red] (2.5,0) circle (1pt);
                \filldraw[red] (0.9823701,3) circle (1pt);

                % linker innerer Kegel äußere Färbung

                \draw[fill=green, opacity=0.2] (0.5,0) arc (360:180:0.5cm and .25cm) -- (-2.428649,5.6945946) arc (194:323:2.5cm and 1.25cm) arc (217:346:2.5cm and 1.25cm) -- (4.5,0) arc (360:180:0.5cm and .25cm) -- (2,4.5) --(0.5,0);

                % innere Kegel innere Färbung

                \draw[fill=green, opacity=0.4] (2,6.75) arc (37:323:2.5cm and 1.25cm) -- (2,5.25)
                arc (217:503:2.5cm and 1.25cm);

                % äußere Kegel äußere Färbung

                \draw[fill=drawColor, opacity=0.2] (-1,0) -- (-3.35882,5.50116) arc (196:305:3.5cm and 1.75cm) -- (2,4.563859) arc (235:344:3.5cm and 1.75cm) -- (5,0) arc (360:180:1cm and .5cm) -- (3,0) -- (2,2.4) -- (1,0) arc (360:180:1cm and .5cm);

                % äußere Kegel innere Färbung

                \draw[fill=drawColor, opacity=0.4] (2,6.75) arc (143:-12:2.5cm and 1.25cm) -- (6.44120,5.73049) -- (6.05228,4.58241735)  arc (-54:125:3.5cm and 1.75cm)  -- (2,7.436141) -- (2,7.436141)  arc (55:234:3.5cm and 1.75cm) -- (-2.05228,4.58241735) -- (-2.44120,5.73049) arc (192:37:2.5cm and 1.25cm);

        \end{tikzpicture}
        \caption{Double cone with a double cone shell (grey), a time slice
                of the middle tube (blue), and a path of a random walk crossing
                the double cone shell from outside to inside (red).}
        \label{fig:dcones}
\end{figure}
Let $(\wt{R}_i)_i$, with $\wt{R}_0 =0$, be the regeneration times for the pair $(\wt{X},\wt{X}')$, obtained via the construction described above. We recall a few of the important properties obtained for this sequence in \cite{BirknerDepperschmidtSchlueter2024}. The first gives a control of the tail probabilities of the regeneration increments, that is
\begin{equation}
        \label{eq:regen_time_tail_bounds}
        \sup_{x_0,x'_0} \bP(\wt{R}_i -\wt{R}_{i-1}> t \,\vert\, X_0 =x_0,X'_0=x'_0) \le Ct^{-\beta},
\end{equation}
where $\beta$ is a positive constant that we can tune arbitrarily large because we can enforce a very high density of `good' boxes in $\eta^{\mathrm{stat}}$. We enrich the probability space by another independent environment $\eta'$ and consequently another ancestral line $X''$ evolving in $\eta'$ with its own coarse-grained version $\wt{X}^{''}$.
A central technique we then employ is to compare the pair $(\wt{X},\wt{X}')$ with $(\wt{X},\wt{X}^{''})$. Since $\wt{X}$ and $\wt{X}^{''}$ evolve in independent environments, they are in fact independent and we introduce regeneration times $(\wt{R}^{\indi}_i)_i$ for $(\wt{X},\wt{X}^{''})$ constructed in an analogous way to the cone construction for $(\wt{R}_i)_i$ but with the cone for $\wt{X}^{''}$ constructed in $\eta'$ instead of $\eta$. It is easy to see that the tails for the increments of $(\wt{R}_i)_i$ satisfy the same bound from \eqref{eq:regen_time_tail_bounds}.
From the description of the cone construction we recall that each cone isolates the environment on the inside from the outside. An intuitive application of this is that, as long as the regeneration happens sufficiently fast such that the cones remain well separated and in particular do not overlap, the increment for $(\wt{X},\wt{X}')$ should have almost the same distribution as the increment of $(\wt{X},\wt{X}^{''})$. Indeed we obtain the following coupling result
\begin{align}
        & \sum_{x \in \Z^d,x'\in \Z^d,m\in \N} \left|
        \bP\left( (X_{R^\joint_1}, X'_{R^\joint_1},R^\joint_1)=(x,x',m)\right)
        - \bP\left( (X_{R^\indi_1}, X''_{R^\indi_1},R^\indi_1)=(x,x',m)\right)
        \right| \notag \\
        \label{eq:TVdistance-joint-ind-1step in prop}
        & \hspace{2em} \le C\norm{x_0-x'_0}^{-\beta},
\end{align}
with $\beta$ being the some constant from \eqref{eq:regen_time_tail_bounds}. This coupling allows to transfer some asymptotic properties from the pair $(\wt{X},\wt{X}^{''})$ to $(\wt{X},\wt{X}')$.

There are other properties we can derive from this coupling result. The first we want to highlight is what we call a separation lemma, which allows us control the time two random walkers need to move apart to some distance. We denote by 
\begin{equation}
    \widehat{X}_k \coloneqq \wt{X}_{\wt{R}_k}, 
    \quad \widehat{X}'_k \coloneqq \wt{X}'_{\wt{R}_k}, 
    \qquad k \in \bN_0
\end{equation}
the walks observed along their joint regeneration times.

For $r>0$ we define $\widehat{H}(r)$ as the first time the coarse-grained random walks along their regeneration sequence are at a distance of at least $r$
\begin{align*}
        \widehat{H}(r) \coloneqq \inf\Big\{ k\in\bZ_+ \colon
        \norm{\widehat{X}_k - \widehat{X}'_k}_2 \ge r \Big\}
\end{align*}
and, for the other direction, let $\widehat{h}(r)$ be the first time the two walkers come together to a distance of $r$, that is
\begin{equation*}
        \widehat{h}(r) \coloneqq \inf\Big\{ k\in\bZ_+ \colon
        \norm{\widehat{X}_k - \widehat{X}'_k}_2 \le r \Big\}.
\end{equation*}
\begin{lemma}[Separation lemma]
        Let $d\ge 2$. For \label{lem:abstract separ} all small enough $\delta>0$ and $\varepsilon>0$ there exist $C,c>0$ such that
        \begin{equation}
                \label{eq:abstract sep d>=2}
                \sup_{x_0,x'_0}\bP^\joint_{x_0, x_0'}\Big( \widehat{H}(\tilde{n}^\delta) > \tilde{n}^{2\delta+\varepsilon} \Big) \le \exp(-C\tilde{n}^c).
        \end{equation}
\end{lemma}

An important technique for the proofs is a delicate interplay of the distance between $\widehat{X}$ and $\widehat{X}'$ and the control over the behaviour we can employ. Another vital component for the proofs below is Lemma~\hittingLemma{} from \cite{BirknerDepperschmidtSchlueter2024}, there proved for a toy model but, using Remark~\abstractPropRem{} from \cite{BirknerDepperschmidtSchlueter2024}, we obtain the following lemma.
\begin{lemma}[Hitting probabilities for spheres]
        \label{lemma:exitAnnulus}
        Put for $r_1<r<r_2$
        \begin{align}
                f_d(r;r_1,r_2)= \begin{cases}
                        \frac{r-r_1}{r_2-r_1}, & \text{when }d=1,\\
                        \frac{\log r - \log r_1}{\log r_2 - \log r_1}, \quad & \text{when }d=2,\\
                        \frac{r_1^{2-d}-r^{2-d}}{r_1^{2-d}-r_2^{2-d}}, \quad & \text{when }d\geq 3.
                \end{cases}
        \end{align}
        For every $\varepsilon >0$ there are (large) $R$ and $\wt{R}$ such
        that for all $r_2>r_1>R$ with $r_2-r_1>\wt{R}$ and
        $x,y \in \bZ^d$ satisfying $r_1 < r=\norm{x-y}_2<r_2$
        \begin{align}
                (1-\varepsilon)f_d(r;r_1,r_2) \leq \Pr^{\indi}_{x,y}\left(\widehat{H}(r_2)<\widehat{h}(r_1)\right)
                \leq (1+\varepsilon)f_d(r;r_1,r_2).
        \end{align}
\end{lemma}

\section{Details for and proof of Theorem~\ref{thm:Pct}
and Corollary~\ref{cor:Malecot}}
\label{sect:main.proof}

\subsection{Tails for meeting times}
\label{sect:Tails}
Consider two particles which start at a distance of $\norm{x}N$, where $x\in\bR^2 \setminus \{ 0 \}$ and $N\in\bN$. Recall the notation introduced in Remark~\ref{rem:defn_prop_measures}.
We are interested in the asymptotic behaviour of the meeting time for the two particles, more specifically we want to show the following proposition.
\begin{proposition}
        \label{prop:limi_meeting_time}
        For $x\in\bR^2 \setminus \{ 0 \}$ and $\gamma>1$ the following limit holds
        \begin{equation}
                \label{eq:limit_meeting_time}
                \lim\limits_{N\to \infty} \bP_{Nx}\big( \tau_{\mathrm{meet}} >N^{2\gamma}  \big) = \frac{1}{\gamma}.
        \end{equation}
\end{proposition}

\begin{remark}
\label{rem:separate}
        Our approach to proving the above proposition will show that there is essentially only one possibility for the event $\{\tau_{\mathrm{meet}} >N^{2\gamma}\}$ to occur and that is, when the two random walkers separate to a distance of $N^\gamma$ before meeting. From equation~\eqref{eq:AH2} we see that this takes $N^{2\gamma}$ many steps. Thus we morally also show that
        \begin{equation*}
                \bP_{Nx} \big( \max_{i \le \tau_{\mathrm{meet}}} \norm{X_i-X'_i} \ge
                N^\gamma \big) \to 1/\gamma.
        \end{equation*}
\end{remark}

The first step toward proving Proposition~\ref{prop:limi_meeting_time} is to consider the behaviour of the coarse-grained pair $(\widehat{X},\widehat{X}')$ introduced in Section~\ref{sect:joint regeneration construction} and wait for those to meet. This gives us a time at which we can ensure that $X$ and $X'$ are at a constant distance. For this reason we start by proving
\begin{align}
  \label{eq:main_claim}
  \lim\limits_{N\to\infty} \bP_{Nx}\big( \widehat{\tau}_{\mathrm{meet}} > N^{2\gamma} \big) = \frac{1}{\gamma}, \qquad \gamma \ge 1,
\end{align}
where we understand $\bP_{Nx}$ as the distribution where one particle is started at the origin and the other at the site $Nx$.

To prove this we use two crucial claims we prove later on. The two particles each perform a random walk (on the coarse-graining level) which we denote by $\widehat{X}$ and $\widehat{X}'$. Since we are interested in the meeting time, it is sufficient to consider the difference of these two and we set $\widehat{D}\coloneqq \widehat{X}-\widehat{X}'$. Moreover we define the stopping times
\begin{equation}
  \label{def:tau_near}
  \widehat{\tau}_{\mathrm{near}} \coloneqq \inf\{ k \colon \norm{\widehat{D}_k} \le M\}
\end{equation}
for some constant $M\in \bR$, and
\begin{equation}
  \label{def:tau_R}
  \widehat{\tau}_{R} \coloneqq \inf\{ k\colon \norm{\widehat{D}_k} \ge R \}
\end{equation}
for some $R\in \bR$.
Note that in the following lemmas we study the behaviour of the coarse-grained pair $(\widehat{X},\widehat{X}')$. While we still write $\bP_{Nx}$ we somewhat sweep under the rug that under $\bP_{Nx}$ the distance between $\widehat{X}$ and $\widehat{X}'$ is actually $N\norm{x}/L_s$ and under $\bP_R$ it is $R/L_s$. However this linear scaling doesn't change the calculations done and we therefore omit this scaling to leave the proofs more accessible to read.
\begin{lemma}
  \label{lem:AHs}
  For $x\in\bR^2 \setminus \{ 0 \}$ and $\gamma>1$ the following two limits hold
  \begin{equation}
    \label{eq:AH1}
    \bP_{Nx} \big( \widehat{\tau}_{\mathrm{near}} > \widehat{\tau}_{N^\gamma}\big) \xrightarrow[N\to\infty]{} \frac{1}{\gamma}
  \end{equation}
  and for $\varepsilon>0$
  \begin{equation}
    \label{eq:AH2}
    \bP_{Nx}\big( N^{2\gamma-\varepsilon} \le \widehat{\tau}_{N^\gamma} \le N^{2\gamma+\varepsilon} \big) \xrightarrow[N\to\infty]{} 1.
  \end{equation}
\end{lemma}
First we show how to apply the lemma to prove the following auxiliary result:
\begin{lemma}
  \label{lem:AH.aux}
  \begin{equation}
    \label{eq:near_limit}
    \bP_{Nx}\big( \widehat{\tau}_{\mathrm{near}} > N^{2\gamma} \big) \xrightarrow[N\to\infty]{} \frac{1}{\gamma}, \qquad \gamma > 1.
  \end{equation}
\end{lemma}
\begin{proof}[Proof of Lemma~\ref{lem:AH.aux}] 
  We start with the lower bound. Note that we have the inclusion $\{\widehat{\tau}_{\mathrm{near}} > \widehat{\tau}_{N^\gamma}\}\cap \{\widehat{\tau}_{N^\gamma}\ge N^{2\gamma-\varepsilon} \} \subset \{ \widehat{\tau}_{\mathrm{near}} \ge N^{2\gamma-\varepsilon} \}$, and therefore
  \begin{align*}
    \liminf\limits_{N\to\infty}\bP_{Nx}\big( \widehat{\tau}_{\mathrm{near}} \ge N^{2\gamma-\varepsilon}\big) &\ge \liminf\limits_{N\to\infty} \bP_{Nx}\big(\widehat{\tau}_{\mathrm{near}} > \widehat{\tau}_{N^\gamma}, \widehat{\tau}_{N^\gamma}\ge N^{2\gamma-\varepsilon}\big)\\
    &\ge \liminf\limits_{N\to\infty} \bP_{Nx}\big(\widehat{\tau}_{\mathrm{near}} > \widehat{\tau}_{N^\gamma}\big) - \bP_{Nx}\big(\widehat{\tau}_{N^\gamma} < N^{2\gamma-\varepsilon}\big)\\
    &\ge \liminf\limits_{N\to\infty} \bP_{Nx}\big(\widehat{\tau}_{\mathrm{near}} > \widehat{\tau}_{N^\gamma}\big) - \limsup\limits_{N\to\infty}
    \bP_{Nx}\big(\widehat{\tau}_{N^\gamma} < N^{2\gamma-\varepsilon}\big)\\
                                                   &\ge \frac{1}{\gamma}
  \end{align*}
  where the last line holds by \eqref{eq:AH1} and \eqref{eq:AH2}. Setting $\tilde{\gamma}=\gamma-\varepsilon/2$ the last display can be written as
  \begin{equation*}
    \liminf\limits_{N\to\infty}\bP_{Nx}\big( \widehat{\tau}_{\mathrm{near}} \ge N^{2\tilde{\gamma}}\big) \ge \frac{1}{\tilde{\gamma}+\varepsilon/2}.
  \end{equation*}
  Note that, since $\gamma>1$, we can choose $\varepsilon$ small enough such that $\tilde{\gamma}>1$.
  Now, letting $\varepsilon$ go to $0$, we obtain
  \begin{align*}
    \liminf\limits_{N\to\infty}\bP_{Nx}\big( \widehat{\tau}_{\mathrm{near}} > N^{2\tilde{\gamma}}\big) \ge \frac{1}{\tilde{\gamma}},
  \end{align*}
  for every $\tilde{\gamma}>1$.\\
  \medskip\noindent
  Turning to the upper bound we note that
  \begin{equation*}
    \{ \widehat{\tau}_{\mathrm{near}}>N^{2\gamma} \} = \{ \widehat{\tau}_{\mathrm{near}}>N^{2\gamma} \ge \tau_{N^{\gamma-\varepsilon}} \} \dot{\cup} \{ \widehat{\tau}_{\mathrm{near}}, \widehat{\tau}_{N^{\gamma-\varepsilon}} > N^{2\gamma}\}.
  \end{equation*}
  For the second set on the right hand side we see
  \begin{equation*}
    \bP_{Nx}\big( \widehat{\tau}_{\mathrm{near}}, \widehat{\tau}_{N^{\gamma-\varepsilon}} > N^{2\gamma} \big) \le  \bP_{Nx}\big( \widehat{\tau}_{N^{\gamma-\varepsilon}}  > N^{2\gamma}\big) \xrightarrow[N\to\infty]{} 0
  \end{equation*}
  by \eqref{eq:AH2}. For the first set we see
  \begin{align*}
    \bP_{Nx}\big( \widehat{\tau}_{\mathrm{near}}>N^{2\gamma} \ge \tau_{N^{\gamma-\varepsilon}} \big) \le \bP_{Nx}\big( \widehat{\tau}_{\mathrm{near}} > \widehat{\tau}_{N^{\gamma-\varepsilon}} \big) \xrightarrow[N\to \infty]{} \frac{1}{\gamma-\varepsilon}
  \end{align*}
  by \eqref{eq:AH1}. Combining the above we obtain
  \begin{equation*}
    \limsup\limits_{N\to\infty} \bP_{Nx}\big( \widehat{\tau}_{\mathrm{near}} > N^{2\gamma} \big) \le \frac{1}{\gamma-\varepsilon}
  \end{equation*}
  and consequently
  \begin{equation*}
    \limsup\limits_{N\to\infty} \bP_{Nx}\big( \widehat{\tau}_{\mathrm{near}} > N^{2\gamma} \big) \le \frac{1}{\gamma},
  \end{equation*}
  which concludes the proof of \eqref{eq:near_limit}.
\end{proof}
Now we prove the two equations from Lemma~\ref{lem:AHs}. 
\begin{proof}[Proof of Lemma~\ref{lem:AHs}]
  We start with proving \eqref{eq:AH1}.
  Define the stopping time $\widehat{\tau}_{\le m_N} \coloneqq \inf\{k\colon \norm{\widehat{D}_k} \le m_N \}$, where $m_N=\sqrt{\log\log N}$ and the sets $A_N\coloneqq \{ \widehat{\tau}_{\le m_N}  > \widehat{\tau}_{N^\gamma} \}$ and $B_N\coloneqq \{ \widehat{\tau}_{\mathrm{near}} > \widehat{\tau}_{N^\gamma} \}$. By definition we have $A_N\subset B_N$ and thus by Lemma~\ref{lemma:exitAnnulus}
  \begin{equation}
    \label{eq:lower_bound_for_AH1}
    \liminf\limits_{N\to\infty} \bP_{Nx}\big( \widehat{\tau}_{\mathrm{near}} > \widehat{\tau}_{N^\gamma} \big) \ge \liminf\limits_{N\to\infty} \bP_{Nx}\big( \widehat{\tau}_{\le m_N}  > \widehat{\tau}_{N^\gamma} \big) = \frac{1}{\gamma}.
  \end{equation}
  To prove the upper bound it is sufficient to show that
  \begin{align}
    \label{eq:1.6}
    \bP_{Nx}(B_N\cap A^\compl_N) =\bP_{Nx}(\widehat{\tau}_{\le m_N} < \widehat{\tau}_{N^\gamma} < \widehat{\tau}_{\mathrm{near}} ) \xrightarrow{} 0.
  \end{align}

  We provide in the following a somewhat rough sketch of the argument for \eqref{eq:1.6}, where we pretend that $\widehat{D} = \widehat{X}-\widehat{X}'$ has ``continuous'' paths in the sense that we use $\widehat{D}_{\widehat{\tau}_R}=R$ and $\widehat{D}_{\widehat{\tau}_{\le m_N}} = m_N$ in the computations. In the full proof, we control the error incurred by over-/undershoots at stopping times explicitly. It is however a little cumbersome and we relegate the details to Appendix~\ref{sect:proof of 1.6}.

  Note that, due to the Markov-property the left-hand side is roughly bounded by $\bP_{m_N}(\widehat{\tau}_{N^\gamma} < \widehat{\tau}_{\mathrm{near}})$. We then use a similar idea to the separation lemma, that is Lemma~\ref{lem:abstract separ}, and introduce a \glqq bridge state\grqq\ at $\wt{m}_N = \log N$ between $m_N$ and $N^\gamma$. Every time the random walks are at distance $m_N$ they have a positive chance, still depending on $N$, to reach a constant distance $M$. Here we can construct corridors, similar to the separation lemma, that enforce a positive probability to reach the distance $M$ within $m_N$ steps. Therefore there exist a positive constant $\delta>0$ such that the distance $M$ can be reached from the distance $m_N$ within $m_N$ steps with probability at least $\delta^{m_N}$. Whereas, for the other distances we obtain, using Lemma~\ref{lemma:exitAnnulus},
  \begin{align}
    \label{eq:5}
    \bP_{\wt{m}_N} ( \widehat{\tau}_{\le m_N} > \widehat{\tau}_{N^\gamma} ) \sim \frac{\log \wt{m}_N - \log m_N}{\log N^\gamma - \log m_N} \sim \frac{\log \log N}{\gamma \log N}.
  \end{align}
  We can write
  \begin{equation*}
    \bP_{m_N}(\widehat{\tau}_{N^\gamma} < \widehat{\tau}_{\mathrm{near}}) =     \bP_{m_N}(\widehat{\tau}_{\wt{m}_N}<\widehat{\tau}_{\mathrm{near}}) \Big(\bP_{\wt{m}_N}( \widehat{\tau}_{N^\gamma} < \widehat{\tau}_{\le m_N} ) + \bP_{\wt{m}_N}( \widehat{\tau}_{\le m_N}<\widehat{\tau}_{N^\gamma} < \widehat{\tau}_{\mathrm{near}} )  \Big).
  \end{equation*}
  and observe $\bP_{m_N}(\widehat{\tau}_{\wt{m}_N}<\widehat{\tau}_{\mathrm{near}}) \le 1-\delta^{m_N}$. We split the attempts to hit the distance $N^\gamma$ by how often the distance is $\wt{m}_N$ before eventually hitting $N^\gamma$. Using this we obtain a sum over the number of failed attempts to reach a constant distance of $M$ starting from $m_N$ as an upper bound, i.e.\
  \begin{align}
    \label{eq:6}
    \begin{split}
      \bP_{m_N}(\widehat{\tau}_{N^\gamma} < \widehat{\tau}_{\mathrm{near}})
      &\le (1-\delta^{m_N}) \bP_{\wt{m}_N}( \widehat{\tau}_{N^\gamma} < \widehat{\tau}_{\le m_N} ) +(1-\delta^{m_N}) \bP_{\wt{m}_N}( \widehat{\tau}_{\le m_N}<\widehat{\tau}_{N^\gamma} < \widehat{\tau}_{\mathrm{near}} )\\
      &\le (1-\delta^{m_N}) \frac{\log\log N}{\gamma \log N} + (1-\delta^{m_N})\bP_{m_N}(\widehat{\tau}_{N^\gamma} < \widehat{\tau}_{\mathrm{near}})\\
      &\le \sum_{k=1} (1-\delta^{m_N})^k \frac{\log \log N}{\gamma \log N}\\
      &\le \frac{\log\log N}{\delta^{m_N}\gamma\log N}.
    \end{split}
  \end{align}
  Note that by the choice of $m_N$ we have
  \begin{equation*}
    \frac{\log\log N}{\delta^{m_N}\gamma\log N} \xrightarrow[N\to \infty]{}0.
  \end{equation*}
  Thus
  \begin{align}
    \lim\limits_{N\to\infty} \bP_{Nx}(\widehat{\tau}_{\mathrm{near}}> \widehat{\tau}_{N^\gamma}) = \frac{1}{\gamma}.
  \end{align}
  See Appendix~\ref{sect:proof of 1.6} for more details.
  \medskip

  For \eqref{eq:AH2} we first observe that, by the separation lemma,
  \begin{align*}
    \bP_{Nx}(\widehat{\tau}_{N^\gamma} \ge N^{2\gamma + \varepsilon}) \le \exp(-C N^{c})
  \end{align*}
  for some constants $C,c>0$.

  For the lower bound, we again give here a rough argument pretending that paths are ``continuous''
  and refer to Appendix~\ref{sect:lb AH2} for more details.
  Since  $N<N^\gamma/2$ for $N$ large enough, we obtain
  \begin{align*}
    \bP_{Nx}(\widehat{\tau}_{N^\gamma} \le N^{2\gamma-\varepsilon}) \le \bP_{N^{\gamma}/2} (\widehat{\tau}_{N^\gamma} \le N^{2\gamma-\varepsilon}).
  \end{align*}
  Furthermore let $\tau_{\mathrm{exit}} \coloneqq \inf\{k\ge 0\colon \norm{\widehat{D}_k} \notin B_{N^\gamma/4}(N^\gamma/2) \}$, i.e.\ the exit time of the ball with radius $N^\gamma/4$ centred at $N^\gamma/2$. This yields
  \begin{align*}
    \bP_{N^{\gamma}/2}(\widehat{\tau}_{N^\gamma} \le N^{2\gamma-\varepsilon}) \le \bP_{N^{\gamma/2}}(\tau_{\mathrm{exit}} \le N^{2\gamma-\varepsilon}).
  \end{align*}
  It remains to analyze the probability on the right hand side of the last display. To that end, note that $\norm{\widehat{D}_k}\ge N^\gamma/4$ for all $k\le \tau_{\mathrm{exit}}$ if $\norm{\widehat{D}_0} \in B_{N^\gamma/4}(N^\gamma/2)$ and we can couple the $\joint$-pair, that is where both random walkers move in the same environment, with a pair of random walkers that evolve in independent copies of the environment, denoted by $\widehat{X}^\indi$ and $\widehat{X}^{'\indi}$. Respectively we denote their difference process by $\widehat{D}^\indi$ and obtain, by using equation~\eqref{eq:TVdistance-joint-ind-1step in prop},
  \begin{align*}
    \bP_{N^{\gamma/2}}(\tau_{\mathrm{exit}} \le N^{2\gamma-\varepsilon}) \le CN^{2\gamma-\varepsilon}N^{-\beta\gamma} + \bP_{N^{\gamma/2}}(\tau^\indi_{\mathrm{exit}} \le N^{2\gamma-\varepsilon}).
  \end{align*}
  By tuning the parameters of the regeneration construction developed in \cite{BirknerDepperschmidtSchlueter2024} correctly we can choose $\beta$ arbitrarily large and, therefore, the first term on the right hand side vanishes in the limit. For the $\indi$-pair we show in \cite{BirknerDepperschmidtSchlueter2024} that a functional CLT holds and thus for any $K>0$
  \begin{align*}
    \bP_{N^{\gamma}/2}(\tau^\indi_{\mathrm{exit}} \le N^{2\gamma-\varepsilon} )\le \bP_{N^{\gamma}/2}\Big( \tau^\indi_{\mathrm{exit}} \le \frac{N^{2\gamma}}{K} \Big) \xrightarrow[N\to \infty]{} \bP_0\Big(\inf\{t\colon \norm{B_t} > 1/4 \} \le \frac{1}{K} \Big).
  \end{align*}
  Thus, by taking $K$ to infinity on both sides in the above display,
  \begin{equation*}
    \limsup\limits_{N\to\infty}\bP_{N^{\gamma/2}}(\tau^\indi_{\mathrm{exit}} \le N^{2\gamma-\varepsilon}) \xrightarrow[N\to\infty]{} 0,
  \end{equation*}
  and
  \begin{equation*}
    \limsup\limits_{N\to\infty}\bP_{Nx}(\widehat{\tau}_{N^\gamma} \le N^{2\gamma-\varepsilon}) =0.
  \end{equation*}
  Finally we conclude that \eqref{eq:AH2} holds.
\end{proof}

Since we aim to study the meeting time of the two original random walks and not just their coarse-grained versions, we need to establish some control of the behaviour between the regeneration times. Our strategy is to observe along the coarse-grained versions along their simultaneous regeneration times and wait for those to hit the same box of some constant size, before we estimate the probability for the two original random walks to meet when started in the same box. Therefore we want to avoid a scenario where the original random walks are near each other and possibly meet before the coarse-grained versions along their regeneration times hit the same box.

Thus, to prove Proposition~\ref{prop:limi_meeting_time}, we need some estimates that allow us to swap from the coarse-grained walkers to the original ones. A lot follows along the same ideas we used above.
\begin{lemma}
  \label{lem:limit_with_regen_times}
  For $x\in\bR^2 \setminus \{ 0 \}$ and $\gamma>1$ the following limit holds
  \begin{equation}
    \label{eq:limit_with_regen_times}
    \lim\limits_{N\to \infty} \bP_{Nx}\Big( \Lt \sum_{i=1}^{\widehat{\tau}_{\mathrm{near}}} \big( \wt{R}_i-\wt{R}_{i-1} \big)>N^{2\gamma} \Big) = \frac{1}{\gamma}.
  \end{equation}
\end{lemma}
\begin{proof}
  We start with the lower bound.
  \begin{align*}
    \bP_{Nx}\Big( \Lt\big( \sum_{i=1}^{\widehat{\tau}_{\mathrm{near}}} \wt{R}_i-\wt{R}_{i-1} \big)>N^{2\gamma} \Big) \ge \bP_{Nx}\big( \widehat{\tau}_{\mathrm{near}}\Lt > N^{2\gamma} \big) \xrightarrow[N\to\infty]{} \frac{1}{\gamma},
  \end{align*}
  by equation~\eqref{eq:near_limit}, where $(\wt{R}_i)_i$ are the simultaneous regeneration times for the coarse-grained pair of random walkers (note that the pre-factor $\Lt$, which is not present in \eqref{eq:near_limit}, is irrelevant in the limit).

  Proving the upper bound essentially relies on the tail bounds proved for the simultaneous regeneration times $(\wt{R}_i)_i$. Let $\varepsilon>0$,
  \begin{align*}
    \bP_{Nx}\Big( \Lt \sum_{i=1}^{\widehat{\tau}_{\mathrm{near}}} \wt{R}_i-\wt{R}_{i-1} > N^{2\gamma}\Big)&= \bP_{Nx}\Big( \Lt \sum_{i=1}^{\widehat{\tau}_{\mathrm{near}}} \wt{R}_i-\wt{R}_{i-1} > N^{2\gamma}, \widehat{\tau}_{\mathrm{near}}\le \varepsilon N^{2\gamma} \Big)\\
    &\hspace{1.5cm}+\bP_{Nx}\Big( \Lt \sum_{i=1}^{\widehat{\tau}_{\mathrm{near}}} \wt{R}_i-\wt{R}_{i-1} > N^{2\gamma}, \widehat{\tau}_{\mathrm{near}}> \varepsilon N^{2\gamma} \Big)\\
    &\le \bP_{Nx}\Big( \Lt \sum_{i=1}^{\varepsilon N^{2\gamma}} \wt{R}_i-\wt{R}_{i-1} > N^{2\gamma} \Big) + \bP_{Nx}\Big( \widehat{\tau}_{\mathrm{near}} > \varepsilon N^{2\gamma} \Big).
  \end{align*}

  Using Markov's inequality the first term in the last line can be bounded from above by
  \begin{align}
    \begin{split}
      \bP_{Nx}\Big( \Lt \sum_{i=1}^{\varepsilon N^{2\gamma}} \wt{R}_i-\wt{R}_{i-1} > N^{2\gamma} \Big)
      & \le \frac{\Lt}{N^{2\gamma}} \sum_{i=1}^{\varepsilon N^{2\gamma}} \bE_{Nx}[\wt{R}_i-\wt{R}_{i-1}]
      \le C \Lt \varepsilon
    \end{split}
  \end{align}
  with some fixed constant $C < \infty$ by \eqref{eq:regen_time_tail_bounds}.
  Thus
  \begin{align}
    \label{eq:1.13}
    \begin{split}
      \limsup_{N\to\infty} \bP_{Nx}\Big( \Lt \sum_{i=1}^{\widehat{\tau}_{\mathrm{near}}} \wt{R}_i-\wt{R}_{i-1} > N^{2\gamma}\Big)
      \le C \Lt \varepsilon + \limsup_{N\to\infty} \bP_{Nx}\Big( \widehat{\tau}_{\mathrm{near}} > \varepsilon N^{2\gamma} \Big)
      \le C \Lt \varepsilon + \frac{1}{\gamma}
    \end{split}
  \end{align}
  by Lemma~\ref{lem:AH.aux} (again, the prefactor $\varepsilon$ in the second term on
  the right-hand side, which is not present in \eqref{eq:near_limit},
  plays no role in the limit). Taking $\varepsilon \downarrow 0$ in \eqref{eq:1.13} yields the
  required upper bound.
\end{proof}

\begin{lemma}
  \label{lem:Greenfct.annulus}
  There are constants $0 < C_1, C_2 < \infty$ such that
  \begin{equation}
    \bE_{Nx}\bigg[ \sum_{i=1}^{\widehat{\tau}_{N^\gamma} \wedge \widehat{\tau}_{\mathrm{near}} - 1} \norm{\widehat{D}_i}_2^{-\beta} \bigg]
    \le C_1 M^{2-\beta} - C_2 \norm{Nx}_2^{2-\beta}
  \end{equation}
  whenever $M$ is large enough (recall from \eqref{def:tau_near} that
  $M$ also enters the definition of $\widehat{\tau}_{\mathrm{near}}$).
\end{lemma}
\begin{proof} We wish to find a function $f : \Z^2 \to [0,\infty)$
  and a constant $c>0$ such that
  \begin{align}
    \label{eq:Poissonproblemannulus}
    \bE\big[ f(\widehat{D}_{i+1})-f(\widehat{D}_i) \,\big|\, \widehat{\mathcal{F}}_i \big]
    \le -c \norm{\widehat{D}_i}_2^{-\beta}
    \quad \text{ on the event } M < \norm{\widehat{D}_i}_2 < N^{2\gamma}
  \end{align}
  and $f(y) \ge 0$ when $\norm{y} \le M$ or $\norm{y} \ge N^{2\gamma}$.
  
  Given $f$ satisfying \eqref{eq:Poissonproblemannulus}, writing
  $\widehat{\tau} := \widehat{\tau}_{N^\gamma} \wedge
  \widehat{\tau}_{\mathrm{near}}$, we have for
  $M < \norm{y} < N^{2\gamma}$
    \begin{align}
    0
    & \le \bE_y\big[ f(\widehat{D}_{\widehat{\tau}})\big]
      = f(y) + \bE_y\bigg[ \sum_{i=0}^{\widehat{\tau}-1} \Big( f(\widehat{D}_{i+1})-f(\widehat{D}_i) \Big) \bigg]
      \notag \\
    &
      = f(y) + \sum_{i=0}^\infty \bE_y\bigg[ \ind{i < \widehat{\tau}} \Big( f(\widehat{D}_{i+1})-f(\widehat{D}_i) \Big) \bigg]
      \notag \\
    & = f(y) + \sum_{i=0}^\infty \bE_y\bigg[ \ind{i < \widehat{\tau}}
      \bE_y\Big[ f(\widehat{D}_{i+1})-f(\widehat{D}_i) \,\Big|\, \widehat{\mathcal{F}}_i \Big] \bigg]
         \le f(y) - c \sum_{i=0}^\infty \bE_y\Big[ \ind{i < \widehat{\tau}} \norm{\widehat{D}_i}_2^{-\beta} \Big] ,
  \end{align}
  hence
  \begin{align}
    \bE_{y}\bigg[ \sum_{i=1}^{\widehat{\tau}_{N^\gamma} \wedge \widehat{\tau}_{\mathrm{near}} - 1} \norm{\widehat{D}_i}_2^{-\beta} \bigg]
    \le \frac{1}{c} f(y) .
  \end{align}

  Fairly straightforward, though somewhat lengthy computations show that
  the ansatz function
  \begin{equation}
    \label{eq:Poissonproblemannulus.ansatz}
    f(x) = 2^{\beta-2} M^{2-\beta} - \norm{x}_2^{2-\beta}
    \indset{[M/2,\infty)}(\norm{x}_2), \quad x \in \Z^2
  \end{equation}
  (which is inspired by the corresponding Green function of
  $2$-dimensional Brownian motion in an annulus, see
  Remark~\ref{rem:BMGreenfct.annulus}) does satisfy
  \eqref{eq:Poissonproblemannulus} with a suitable choice of $c>0$.
  For details of the computation see
  Section~\ref{sect:Greenfct.annulus}.
\end{proof}

\begin{proof}[Proof of Proposition~\ref{prop:limi_meeting_time}]
  We first establish a lower bound on the probability on the left-hand side of \eqref{eq:limit_meeting_time}.
  Note that by construction, when $c_0>0$ is sufficiently small,
  \begin{equation}
    \begin{split}
    & \{ N^{2\gamma-\varepsilon} \le \widehat{\tau}_{N^\gamma} \le N^{2\gamma+\varepsilon} \} \cap
    \{ \widehat{\tau}_{N^\gamma} < \widehat{\tau}_{\mathrm{near}} \}
    \cap \{ \wt{R}_{i+1} - \wt{R}_i < c_0 \norm{\widehat{D}_i} \text{ for all } 0 \le i < \widehat{\tau}_{N^\gamma} \} \\
    & \subset \{ \tau_{\mathrm{meet}} > \Lt \widetilde{R}_{N^{2\gamma-\varepsilon}} \}
    \end{split}
  \end{equation}
  (note that at time $\Lt \wt{R}_i$, the two walks have a distance of
  at least $\Ls (\norm{\widehat{D}_i}-1)$ and until the next
  simultaneous regeneration, their distance can decrease by at most
  $\Lt (\wt{R}_{i+1} - \wt{R}_i) 2 R_X$).  Now
  \begin{align}
    & \bP_{Nx}\big( N^{2\gamma-\varepsilon} \le \widehat{\tau}_{N^\gamma} \le N^{2\gamma+\varepsilon}, \,
    \widehat{\tau}_{N^\gamma} < \widehat{\tau}_{\mathrm{near}}, \, \text{there exists } 0 \le i < \widehat{\tau}_{N^\gamma}
      \text{ with } \wt{R}_{i+1} - \wt{R}_i \ge c_0 \norm{\widehat{D}_i} \big) \notag \\
    & \le \bE_{Nx}\Big[ \ind{N^{2\gamma-\varepsilon} \le \widehat{\tau}_{N^\gamma} \le N^{2\gamma+\varepsilon}, \,
      \widehat{\tau}_{N^\gamma} < \widehat{\tau}_{\mathrm{near}}} \sum_{i=0}^{\widehat{\tau}_{N^\gamma}-1}
      \ind{\wt{R}_{i+1} - \wt{R}_i \ge c_0 \norm{\widehat{D}_i}} \Big] \notag \\
    & \le \sum_{i=0}^{N^{2\gamma+\varepsilon}} \bE_{Nx}\Big[ \ind{\wt{R}_{i+1} - \wt{R}_i \ge c_0 \norm{\widehat{D}_i}}
      \ind{i < \widehat{\tau}_{N^\gamma} \wedge \widehat{\tau}_{\mathrm{near}}} \Big] \notag \\
    & = \sum_{i=0}^{N^{2\gamma+\varepsilon}} \bE_{Nx}\Big[ \ind{i < \widehat{\tau}_{N^\gamma} \wedge \widehat{\tau}_{\mathrm{near}}}
      \bP_{Nx}\big( \ind{\wt{R}_{i+1} - \wt{R}_i \ge c_0 \norm{\widehat{D}_i}} \,\big|\, \wt{\mathcal{F}}_{\wt{R}_i} \big) \Big]
      \notag \\
    &  \le C \sum_{i=0}^{N^{2\gamma+\varepsilon}} \bE_{Nx}\big[ \ind{i < \widehat{\tau}_{N^\gamma} \wedge \widehat{\tau}_{\mathrm{near}}}
      \norm{\widehat{D}_i}^{-\beta} \big]
      \le C \bE_{Nx}\bigg[ \sum_{i=1}^{\widehat{\tau}_{N^\gamma} \wedge \widehat{\tau}_{\mathrm{near}} - 1} \norm{\widehat{D}_i}^{-\beta} \bigg]
      \le C M^{2-\beta} 
      \label{eq:Greenfct.annulus bound}
  \end{align}
  by Lemma~\ref{lem:Greenfct.annulus}. We can thus estimate as follows: 
  \begin{align}
    & \bP_{Nx}\big( N^{2\gamma-\varepsilon} \le \widehat{\tau}_{N^\gamma} \le N^{2\gamma+\varepsilon}, \,
    \widehat{\tau}_{N^\gamma} < \widehat{\tau}_{\mathrm{near}}, \, \wt{R}_{i+1} - \wt{R}_i < c_0 \norm{\widehat{D}_i}
      \text{ for all } 0 \le i < \widehat{\tau}_{N^\gamma}\big) \notag \\
    & = \bP_{Nx}\big( N^{2\gamma-\varepsilon} \le \widehat{\tau}_{N^\gamma} \le N^{2\gamma+\varepsilon}, \,
      \widehat{\tau}_{N^\gamma} < \widehat{\tau}_{\mathrm{near}}\big) \notag \\
    & \quad - \bP_{Nx}\big( N^{2\gamma-\varepsilon} \le \widehat{\tau}_{N^\gamma} \le N^{2\gamma+\varepsilon}, \,
    \widehat{\tau}_{N^\gamma} < \widehat{\tau}_{\mathrm{near}}, \, \text{there exists } 0 \le i < \widehat{\tau}_{N^\gamma}
      \text{ with } \wt{R}_{i+1} - \wt{R}_i \ge c_0 \norm{\widehat{D}_i} \big) \notag\\
      \label{eq:Frage}
    & \ge \frac{1}{\gamma} -o(1) 
      - C M^{2-\beta}
  \end{align}
  for $N\to \infty$ by Lemma~\ref{lem:AHs} and \eqref{eq:Greenfct.annulus bound}, thus
  \begin{equation}
    \label{eq:1.17}
    \begin{split}
      \liminf_{N\to\infty} \bP_{Nx}\big( \tau_{\mathrm{meet}} > N^{2\gamma-\varepsilon} \big) \ge
      \liminf_{N\to\infty} \bP_{Nx}\big( \tau_{\mathrm{meet}} > \Lt \widetilde{R}_{N^{2\gamma-\varepsilon}} \big)
      \ge \frac{1}{\gamma} - C M^{2-\beta}.
    \end{split}
  \end{equation}
  Replacing $\gamma$ by $\gamma+\varepsilon/2$ in \eqref{eq:1.17}, then taking $\varepsilon \downarrow 0$
  and then $M\to\infty$ in the construction shows the lower bound in \eqref{eq:limit_meeting_time}.
  \smallskip

  For the upper bound in \eqref{eq:limit_meeting_time}, we argue as follows:
  Write
  \begin{align}
    & \bP_{Nx}\big( \tau_{\mathrm{meet}} > N^{2\gamma} \big) \notag \\
    & = \bP_{Nx}\big( \tau_{\mathrm{meet}} > N^{2\gamma}, \, \Lt \wt{R}_{\widehat{\tau}_{\mathrm{near}}} > N^{2\gamma}/2 \big)
      + \bP_{Nx}\big( \tau_{\mathrm{meet}} > N^{2\gamma}, \, \Lt \wt{R}_{\widehat{\tau}_{\mathrm{near}}} \le N^{2\gamma}/2 \big) \notag \\
     \label{eq:upper_bound_1v2}
    & \le \bP_{Nx}\big( \Lt \wt{R}_{\widehat{\tau}_{\mathrm{near}}} > N^{2\gamma}/2 \big)
      + \bP_{Nx}\big( \tau_{\mathrm{meet}} \ge \Lt \wt{R}_{\widehat{\tau}_{\mathrm{near}}} + N^{2\gamma}/2 \big).
  \end{align}
  where we already know for the second term in the last line
  \begin{equation*}
        \lim\limits_{N\to \infty} \bP_{Nx}\big( \Lt\wt{R}_{\widehat{\tau}_{\mathrm{near}}}>N^{2\gamma}/2 \big) = \frac{1}{\gamma}.
  \end{equation*}
  It remains to show that
  \begin{equation}
    \label{eq:upper_bound_2v2}
    \lim\limits_{N\to \infty}\bP_{Nx}\big( \tau_{\mathrm{meet}} \ge \Lt \wt{R}_{\widehat{\tau}_{\mathrm{near}}} + N^{2\gamma}/2 \big) =0.
  \end{equation}
  In fact, the proof we provide for \eqref{eq:upper_bound_2v2} shows that one could prove the stronger claim that for any $\varepsilon>0$ one obtains $\bP_{Nx}( \tau_{\mathrm{meet}} - \Lt \wt{R}_{\widehat{\tau}_{\mathrm{near}}} \ge N^\varepsilon ) \to 0$. We stick to \eqref{eq:upper_bound_2v2} since it is the exact limit we consider here.
  
  At this point we want to give a small sketch of the idea to prove \eqref{eq:upper_bound_2v2}. In essence we want to use a similar iteration we used in the proof of Lemma~\ref{lem:AHs} to get from a distance of $N$ to constant distance that is independent of $N$. To that end define the sequence of stopping times $\widehat{H}_i \coloneqq\inf\{ k>\widehat{H}_{i-1}\colon \norm{\widehat{D}_k}\le M \}$, starting with $\widehat{H}_1 = \widehat{\tau}_{\mathrm{near}}$, at which the coarse-grained random walks along their regeneration times hit a distance of $M$. At every $\widehat{H}_i$ there is a positive chance $\rho(M\Ls)$ independent of $N$ that the original random walkers meet within the next $M\Ls$ steps. And afterwards, if the attempt failed, we wait until the coarse-grained random walks separate to a distance of $\log N$. Thus the attempts are independent of each other and it is sufficient to find a suitable upper bound for the time each attempt takes to ensure there are enough.
  \begin{align}
        \label{eq:7v2}
        \begin{split}
                \bP_{Nx}& \big( \tau_{\mathrm{meet}}- \Lt \wt{R}_{\widehat{\tau}_{\mathrm{near}}} > \frac{1}{2}N^{2\gamma} \big)\\ &\le \bP_{M} \big( \tau_{\mathrm{meet}} > \frac{1}{2}N^{2\gamma} , \wt{R}_{\widehat{H}_{\log^{1/2} N}} \le N^{\gamma} \big) + \bP_{M} \big( \wt{R}_{\widehat{H}_{\log^{1/2} N}} > N^{\gamma} \big).
        \end{split}
  \end{align}
  To control the right hand side of the above display, we determine a sufficient upper bound on the second term. For the first term note that on the event $\{\wt{R}_{\widehat{H}_{\log^{1/2} N}} \le N^{\gamma}\}$ the original random walkers will have at least $\log^{1/2}N$ independent attempts to meet before time $N^{2\gamma}/2$ each having a chance greater than $\rho(M\Ls)>0$ to succeed and thus the first term will go to $0$ for $N\to \infty$.\\
  If an attempt to meet for original random walkers $X$ and $X'$, starting from a distance of $M\Ls$ fails, we want to wait for the coarse-grained pair $\widehat{X}$ and $\widehat{X}'$ to separate to a distance of $\log N$. Again we observe the coarse-grained pair along their simultaneous regeneration times. By the separation lemma, see Lemma~\ref{lem:abstract separ}, we have
  \begin{align}
        \label{eq:2v2}
        \begin{split}
                \bP_{M}&\Big( \sum_{i=1}^{\widehat{\tau}_{\log N}} \wt{R}_i -\wt{R}_{i-1} > \log^4 N \Big)\\
                &\le \bP_{M}\Big( \sum_{i=1}^{\log^3 N} \wt{R}_i -\wt{R}_{i-1} > \log^4 N \Big) + \bP_{M}\Big( \widehat{\tau}_{\log N} > \log^3 N \Big)\\
                &\le \log^3(N) \cdot \log^{-\beta}(N) + \exp(-C\log^c N).
        \end{split}
  \end{align}
  Additionally we need an upper bound for the amount of time it takes to get back to the constant distance $M$ from which another attempt to meet can be made. For a small positive constant $\tilde{c}>0$ we obtain
   \begin{align}
        \label{eq:3v2}
        \begin{split}
                \bP_{\log N}\Big( \sum_{i=1}^{\widehat{\tau}_{\mathrm{near}}} \wt{R}_i-\wt{R}_{i-1}>N^{2\tilde{c}} \Big) &\le \bP_{\log N}\Big( \sum_{i=1}^{N^{\tilde{c}}} \wt{R}_i-\wt{R}_{i-1}>N^{2\tilde{c}} \Big) + \bP_{\log N}\big( \widehat{\tau}_{\mathrm{near}} >N^{\tilde{c}} \big)\\
                &\le N^{\tilde{c}(1-2\beta)} + \bP_{\log N}\big( \widehat{\tau}_{\mathrm{near}} >N^{\tilde{c}} \big)
        \end{split}
  \end{align}
  It remains to find a suitable upper bound for $\bP_{\log N}( \widehat{\tau}_{\mathrm{near}} >N^{\tilde{c}} )$. We do this in two parts, which we obtain on the right hand side of the display below
  \begin{align}
        \label{eq:4v2}
        \bP_{\log N}\big( \widehat{\tau}_{\mathrm{near}} >N^{\tilde{c}} \big) = \bP_{\log N}\big(\widehat{\tau}_{\mathrm{near}} > N^{\tilde{c}}>\widehat{\tau}_{N^{\tilde{c}/2-\varepsilon}}\big) + \bP_{\log N}\big( \widehat{\tau}_{\mathrm{near}}, \widehat{\tau}_{N^{\tilde{c}/2-\varepsilon}}>N^{\tilde{c}} \big),
  \end{align}
  where $\varepsilon>0$ is some small enough constant such that $\tilde{c}/2-\varepsilon>0$.
  Starting with the second term we get by the separation lemma, Lemma~\ref{lem:abstract separ},
  \begin{align*}
        \bP_{\log N}\big( \widehat{\tau}_{\mathrm{near}}, \widehat{\tau}_{N^{\tilde{c}/2-\varepsilon}}>N^{\tilde{c}} \big) &\le \bP_{\log N}\big( \widehat{\tau}_{N^{\tilde{c}/2-\varepsilon}} >N^{\tilde{c}}\big)\\
        &= \bP_{\log N}\big( \widehat{\tau}_{N^{\tilde{c}'}} > N^{2\tilde{c}'+2\varepsilon} \big)\\
        &\le \exp(-CN^{-2c(\tilde{c}'+\varepsilon)}),
  \end{align*}
  where $\tilde{c}'=\tilde{c}/2-\varepsilon$. For the first term in \eqref{eq:4v2} we obtain
   \begin{align}
        \label{eq:8v2}
        \begin{split}
                \bP_{\log N}&\big(\widehat{\tau}_{\mathrm{near}} > N^{\tilde{c}}>\widehat{\tau}_{N^{\tilde{c}/2-\varepsilon}}\big)\\
                &\le \bP_{\log N}\big( \widehat{\tau}_{\mathrm{near}} > \widehat{\tau}_{N^{\tilde{c}/2-\varepsilon}} \big)\\
                &=\bP_{\log N}\big(\widehat{\tau}_{\mathrm{near}} \ge \widehat{\tau}_{\le m_N} > \widehat{\tau}_{N^{\tilde{c}/2-\varepsilon}}  \big) + \bP_{\log N}\big( \widehat{\tau}_{\mathrm{near}} > \widehat{\tau}_{N^{\tilde{c}/2-\varepsilon}} >\widehat{\tau}_{\le m_N} \big)\\
                &\le \bP_{\log N}\big( \widehat{\tau}_{\le m_N} > \widehat{\tau}_{N^{\tilde{c}/2-\varepsilon}}  \big) + \bP_{m_N}\big( \widehat{\tau}_{\mathrm{near}} > \widehat{\tau}_{N^{\tilde{c}/2-\varepsilon}}  \big)\\
                &\le C\frac{\log\log N - \log m_N}{(\frac{\tilde{c}}{2}-\varepsilon)\log N - \log m_N} + C\frac{\log\log N}{\delta^{m_N}(\frac{\tilde{c}}{2}-\varepsilon)\log N},
        \end{split}
  \end{align}
  where the estimates in the last line follow by calculations done in the proof of Lemma~\ref{lem:AHs}; see equations~\eqref{eq:5} and \eqref{eq:6}.

  Turning back to \eqref{eq:7v2} we now have the tools to give an upper bound for the right hand side of the inequality. Since $\wt{R}_{\widehat{H}_{\log^{1/2}N}}$ is the $\widehat{H}_{\log^{1/2}N}$-th time of the above described iteration at which the random walkers can make an attempt at meeting within the next $M\Ls$ steps. Defining $\widehat{\tau}^i_{\log N} \coloneqq \inf\{k > \widehat{H}_{i-1}\colon \norm{\widehat{D}_k} \ge \log N \}$ we get
   \begin{align*}
     \bP_{M}&\big( \wt{R}_{\widehat{H}_i}-\wt{R}_{\widehat{H}_{i-1}}> \log^4 N + N^{2\tilde{c}} \big)\\
            &\le \bP_{M}\big( \sum_{j=\widehat{H}_{i-1}}^{\widehat{\tau}^i_{\log N}} \wt{R}_j-\wt{R}_{j-1} > \log^4 N\big) + \bP_M\big( \sum_{j=\widehat{\tau}^i_{\log N}+1}^{\widehat{H}_i} \wt{R}_j-\wt{R}_{j-1} > N^{2\tilde{c}}\big)\\
            &\le C\log^{-1+\delta}(N),
   \end{align*}
   where we used the upper bounds obtained in \eqref{eq:2v2}, \eqref{eq:3v2} and \eqref{eq:8v2} and the fact that we can choose $\beta$ arbitrarily large. Therefore, for the second term in \eqref{eq:7v2} we have
  \begin{align*}
        \bP_{M}&\big( \wt{R}_{\widehat{H}_{\log^{1/2}N}} > N^\gamma \big)\\
        &\le \bP_{M}\big( \exists\, i\le \log^{1/2}N\colon \wt{R}_{\widehat{H}_{i}}-\wt{R}_{\widehat{H}_{i-1}} > \frac{N^\gamma}{\log^{1/2}N} \big)\\
        &\le \log^{1/2}N \cdot C\log^{-1} N.
  \end{align*}
  In conclusion we have
  \begin{equation*}
        \bP_{Nx} \big( \tau_{\mathrm{meet}}- \Lt \wt{R}_{\widehat{\tau}_{\mathrm{near}}} > \frac{1}{2}N^{2\gamma} \big) \xrightarrow[N\to \infty]{} 0,
  \end{equation*}
 and therefore, returning back to equation~\eqref{eq:upper_bound_1v2}, we obtain the upper bound
\begin{equation*}
        \lim\limits_{N\to \infty} \bP_{Nx}\big( \tau_{\mathrm{meet}} > N^{2\gamma} \big) \le \frac{1}{\gamma},
\end{equation*}
which concludes the proof.
\end{proof}

\subsection{Completion of the proof of Theorem~\ref{thm:Pct}}

We briefly sketch how to complete the proof of Theorem~\ref{thm:Pct}
based on the results from Section~\ref{sect:Tails}: Since
$\tau_{\mathrm{meet}} \le \tau_{\mathrm{coal}}$, one half follows
immediately from Proposition~\ref{prop:limi_meeting_time}, namely
$\liminf_{N\to \infty} \bP_{Nx}\big( \tau_{\mathrm{coal}} >N^{2\gamma}
\big) \ge \frac{1}{\gamma}$.

For the other direction, we argue as above (see the discussion
following \eqref{eq:upper_bound_2v2}) that for any $\varepsilon>0$,
$\lim_{N\to\infty} \bP_{Nx}\big( \tau_{\mathrm{coal}} -
\tau_{\mathrm{meet}} \ge N^\varepsilon \big) = 0$.  (In fact, one
could even show that
$\lim_{T\to \infty} \inf_{y \in \Z^2} \bP_y\big( \tau_{\mathrm{coal}}
- \tau_{\mathrm{meet}} \le T \big) = 1$.)

\subsection{Asymptotics for the Laplace transform of pair coalescence times:
  Proof of Corollary~\ref{cor:Malecot}}

\begin{proof}[Proof of \eqref{eq:LT_coalescence_time}, using \eqref{eq:limit_coalescence_time}]
  Fix $\gamma \ge 1$.
  We have for $\mu_N = m N^{-2\gamma}$
  \[
    \bE_{Nx}\big[ (1-\mu_N)^{2 \tau_{\mathrm{coal}}} \big] \sim \bE_{Nx}\big[ \exp(- 2 \mu_N \tau_{\mathrm{coal}}) \big].
  \]

  For $0 < \varepsilon < 2\gamma$ we have for all sufficiently large $N$
  \begin{align*}
    \exp(- N^{-\varepsilon}) \bP_{Nx}\big( \tau_{\mathrm{coal}} \le N^{2\gamma-\varepsilon}\big)
    \le \bE_{Nx}\big[ \exp(- m N^{-2\gamma} \tau_{\mathrm{coal}}) \big]
    \le \bP_{Nx}\big( \tau_{\mathrm{coal}} \le N^{2\gamma+\varepsilon}\big) + \exp(- N^{\varepsilon}) .
  \end{align*}
  Thus, \eqref{eq:limit_coalescence_time} yields
  \begin{align*}
    1 - \frac{1}{\gamma - \varepsilon/2}
    \le \liminf_{N\to\infty} \bE_{Nx}\big[ \exp(- 2\mu_N \tau_{\mathrm{coal}}) \big]
    \le \limsup_{N\to\infty} \bE_{Nx}\big[ \exp(- 2\mu_N \tau_{\mathrm{coal}}) \big]
    \le 1 - \frac{1}{\gamma + \varepsilon/2}.
  \end{align*}
  Taking $\varepsilon \downarrow 0$ gives \eqref{eq:LT_coalescence_time}.
\end{proof}

\begin{appendix}
  \section{Detailed proof of \eqref{eq:1.6}}
  \label{sect:proof of 1.6}
  Recall $ m_N = \sqrt{\log\log N}$ and $\widetilde{m}_N = \log N$. Put
  \begin{equation}
    \label{def:a_N,M}
    a_{N,M} := \sup \big\{ \bP_y(\widehat{\tau}_{N^\gamma} < \widehat{\tau}_{\mathrm{near}}) : M < \norm{y} \le m_N \big\}
  \end{equation}
  and
  \begin{equation}
    \label{def:b_N,M}
    b_{N,M} := \sup \big\{ \bP_z(\widehat{\tau}_{N^\gamma} < \widehat{\tau}_{\mathrm{near}}) : \log N < \norm{z} \le 2 \log N \big\} .
  \end{equation}

  For $y \in \Z^d $ with $M < \norm{y} \le m_N$ we have
  \begin{align}
    \bP_y(\widehat{\tau}_{N^\gamma} < \widehat{\tau}_{\mathrm{near}})
    & = \bE_y\big[ \bP_y(\widehat{\tau}_{N^\gamma} < \widehat{\tau}_{\mathrm{near}} \mid \mathcal{F}_{\widehat{\tau}_{\wt{m}_N}})\big] \notag \\
    & = \bE_y\left[ \ind{\widehat{\tau}_{\wt{m}_N} < \widehat{\tau}_{\mathrm{near}}}
      \bP_{\widehat{D}_{\widehat{\tau}_{\wt{m}_N}}}(\widehat{\tau}_{N^\gamma} < \widehat{\tau}_{\mathrm{near}}) \right]  \notag\\
    & \le \bE_y\left[ \ind{\widehat{\tau}_{\wt{m}_N} < \widehat{\tau}_{\mathrm{near}}} \big( b_{N,M} + C (\log N)^3(\log N)^{-\beta} \big)
      \right] \notag \\
    & = \bP_y(\widehat{\tau}_{\wt{m}_N} < \widehat{\tau}_{\mathrm{near}}) \big( b_{N,M} + C (\log N)^3(\log N)^{-\beta} \big)
      \label{eq:A3}
  \end{align}
  where we estimate
  \begin{align*}
    \bP_y(\norm{ \widehat{D}_{\widehat{\tau}_{\wt{m}_N}} } > 2 \log N)
    & \le \bP_y\big(\widehat{\tau}_{\wt{m}_N} > (\log N)^3\big) \\
    & \quad
      + \bP_y\big( \norm{\widehat{D}_k-\widehat{D}_{k-1}} \ge \log N \text{ for some } 1 \le k \le (\log N)^3 \big) \\
    & \le e^{-C(\log N)^c} + C(\log N)^3(\log N)^{-\beta} \\
    &\le  C(\log N)^3(\log N)^{-\beta}
  \end{align*}
  where we used Lemma~\ref{lemma:exitAnnulus} and the tail bound for increments from \eqref{eq:regen_time_tail_bounds} combined with finite range for the transition kernel $p$ in the second inequality.

  Obviously, $\bP_y(\widehat{\tau}_{\widetilde{m}_N} < \widehat{\tau}_{\mathrm{near}}) = 1 - \bP_y(\widehat{\tau}_{\widetilde{m}_N} > \widehat{\tau}_{\mathrm{near}})$.
  By constructing suitable ``corridors'' (as in the proof of Lemma~\separationlemma{} in \cite{BirknerDepperschmidtSchlueter2024}) we can estimate
  \begin{align}
    & \sup \big\{ \bP_y(\widehat{\tau}_{\widetilde{m}_N} < \widehat{\tau}_{\mathrm{near}}) : M < \norm{y} \le m_N \big\} \notag \\
    & = 1 - \inf \big\{ \bP_y(\widehat{\tau}_{\widetilde{m}_N} > \widehat{\tau}_{\mathrm{near}}) : M < \norm{y} \le m_N \big\}
      \le 1 - \delta^{m_N}
  \end{align}
  with some (small) $\delta > 0$.
  Thus, we obtain from \eqref{eq:A3} that
  \begin{equation}
    \label{eq:A5}
    a_{N,M} \le (1 - \delta^{m_N}) \big( b_{N,M} + C (\log N)^3(\log N)^{-\beta} \big)
  \end{equation}

  To gain something from \eqref{eq:A5}, we want to estimate $b_{N,M}$ in terms of $a_{N,M}$.
  Consider $z \in \Z^d$ with $\log N < \norm{z} \le 2\log N$, then
  \begin{align}
    \bP_{z} (\widehat{\tau}_{N^\gamma} < \widehat{\tau}_{\mathrm{near}})
    & = \bP_{z} (\widehat{\tau}_{N^\gamma} < \widehat{\tau}_{\le m_N} \le \widehat{\tau}_{\mathrm{near}})
      + \bP_{z} (\widehat{\tau}_{\le m_N} < \widehat{\tau}_{N^\gamma} < \widehat{\tau}_{\mathrm{near}}) \notag \\
    & \le C \frac{\log\log N}{\log N} + a_{N,M}
      \label{eq:A6}
  \end{align}
  where we used the hitting lemma \cite[Lemma~\hittingLemma{} and Remark~\abstractPropRem{}]{BirknerDepperschmidtSchlueter2024},
  recalled in Lemma~\ref{lemma:exitAnnulus} (with $r=\norm{z}$, $r_1 = m_N, r_2=N^\gamma$) for the
  first term and observe that
  \begin{align*}
    \bP_{z} (\widehat{\tau}_{\le m_N} < \widehat{\tau}_{N^\gamma} < \widehat{\tau}_{\mathrm{near}})
    & = \bE_z\left[ \ind{\widehat{\tau}_{\le m_N} < \widehat{\tau}_{N^\gamma} \wedge \widehat{\tau}_{\mathrm{near}}}
      \bP_{\widehat{D}_{\widehat{\tau}_{\le m_N}}}(\widehat{\tau}_{N^\gamma} < \widehat{\tau}_{\mathrm{near}}) \right]
      \le a_{N,M}.
  \end{align*}
  Taking the supremum of $z$'s with $\log N < \norm{z} \le 2 \log N$ we obtain from
  \eqref{eq:A6} that
  \begin{equation}
    b_{N,M} \le C \frac{\log\log N}{\log N} + a_{N,M}.
    \label{eq:A7}
  \end{equation}

  Inserting \eqref{eq:A7} into \eqref{eq:A5} yields
  \begin{equation}
    a_{N,M} \le (1 - \delta^{m_N}) \Big( a_{N,M} +  C \frac{\log\log N}{\log N} \Big),
    \label{eq:A8}
  \end{equation}
  iterating \eqref{eq:A8} gives
  \begin{equation}
    a_{N,M} \le C \frac{\log\log N}{\log N} \sum_{k=1}^\infty (1 - \delta^{m_N})^k
    \le C \frac{\log\log N}{\log N} \frac{1}{\delta^{m_N}}
    \mathop{\longrightarrow}_{N\to\infty} 0.
  \end{equation}

  To complete the proof of \eqref{eq:1.6} we note that
  \begin{align*}
    \bP_{Nx}(\widehat{\tau}_{\le m_N} < \widehat{\tau}_{N^\gamma} < \widehat{\tau}_{\mathrm{near}} )
    = \bE_{Nx} \left[ \ind{\widehat{\tau}_{\le m_N} <  \widehat{\tau}_{N^\gamma} \wedge \widehat{\tau}_{\mathrm{near}}}
    \bP_{\widehat{D}_{\widehat{\tau}_{\le m_N}}}(\widehat{\tau}_{N^\gamma} < \widehat{\tau}_{\mathrm{near}}) \right]
    \le a_{N,M}.
  \end{align*}
  \hfill $\qed$

  \section{Detailed proof of the lower bound in \eqref{eq:AH2}}
  \label{sect:lb AH2}
  We have
  \begin{align*}
    & \bP_{Nx}(\widehat{\tau}_{N^\gamma} \le N^{2\gamma-\varepsilon}) \\
    & \le \bE_{Nx}\Big[ \ind{\widehat{\tau}_{N^\gamma/2} \le N^{2\gamma-\varepsilon}} \Big(
      \ind{ \norm{\widehat{D}_{\widehat{\tau}_{N^\gamma/2}}} \ge \tfrac34 N^\gamma }
      + \ind{ N^\gamma /2 \le \norm{\widehat{D}_{\widehat{\tau}_{N^\gamma/2}}} < \tfrac34 N^\gamma }
      \bP_{\widehat{D}_{\widehat{\tau}_{N^\gamma/2}}}(\widehat{\tau}_{N^\gamma} \le N^{2\gamma-\varepsilon}) \Big) \Big] \\
    & \le N^{2\gamma-\varepsilon} \big(\tfrac14 N^\gamma\big)^{-\beta}
      + \sup \Big\{ \bP_y(\widehat{\tau}_{N^\gamma} \le N^{2\gamma-\varepsilon}) : \tfrac12 N^\gamma \le \norm{y} \le \tfrac34 N^\gamma \Big\}
  \end{align*}
  and then argue as above.

  \section{Proof details for Lemma~\ref{lem:Greenfct.annulus}}
  \label{sect:Greenfct.annulus}

  \begin{proof}[Proof details for Lemma~\ref{lem:Greenfct.annulus}:
    The function $f$ from \eqref{eq:Poissonproblemannulus.ansatz}
    satisfies \eqref{eq:Poissonproblemannulus}.]
    Recall that we need to verify that the function
    \begin{equation}
      f(x) = 2^{\beta-2} M^{2-\beta} - \norm{x}_2^{2-\beta}
      \indset{[M/2,\infty)}(\norm{x}_2), \quad x \in \Z^2
      \tag{\ref{eq:Poissonproblemannulus.ansatz}}
    \end{equation}
    fulfills
    \begin{align}
      \bE\big[ f(\widehat{D}_{i+1})-f(\widehat{D}_i) \,\big|\, \widehat{\mathcal{F}}_i \big]
      \le -c \norm{\widehat{D}_i}_2^{-\beta}
      \quad \text{ on the event } M < \norm{\widehat{D}_i}_2 < N^{2\gamma}
      \tag{\ref{eq:Poissonproblemannulus}}
    \end{align}
    for some $c>0$.

    In cartesian coordinates, we have
    \[
      f(x_1,x_2) = 2^{\beta-2} M^{2-\beta} - (x_1^2 + x_2^2)^{1-\beta/2} \quad\text{for}\quad
      x_1^2+x_2^2 \ge M^2/4
    \]
    and in this case
    \begin{align}
      \frac{\partial}{\partial x_1}f(x_1,x_2)
      & = 2(\beta/2-1) x_1 (x_1^2 + x_2^2)^{-\beta/2} \\
      \frac{\partial}{\partial x_2}f(x_1,x_2)
      & = 2(\beta/2-1) x_2 (x_1^2 + x_2^2)^{-\beta/2} \\
      \frac{\partial^2}{\partial x_1^2}f(x_1,x_2)
      & = 2(\beta/2-1) (x_1^2 + x_2^2)^{-\beta/2}
        - 4(\beta/2-1)(\beta/2) x_1^2 (x_1^2 + x_2^2)^{-\beta/2-1} \\
      \frac{\partial^2}{\partial x_2^2}f(x_1,x_2)
      & = 2(\beta/2-1) (x_1^2 + x_2^2)^{-\beta/2}
        - 4(\beta/2-1)(\beta/2) x_2^2 (x_1^2 + x_2^2)^{-\beta/2-1} \\
      \frac{\partial^2}{\partial x_1 \partial x_2}f(x_1,x_2)
      & = -4 (\beta/2-1)(\beta/2) x_1 x_2 (x_1^2 + x_2^2)^{-\beta/2-1}
        = \frac{\partial^2}{\partial x_2 \partial x_1}f(x_1,x_2) .
    \end{align}

    Now for $\norm{(x_1,x_2)}_2 \ge M$ we have by Taylor expansion
    (we write $\Psi^\indi(x,y)$ to denote the transition probabilites of $\widehat{D}$ under
    the $\indi$-dynamics)
    \begin{align}
      & \sum_{y \in \Z^2} \Psi^\indi(x,y)\big( f(y) -f(x)\big) \notag \\
      & = \sum_{y \, : \, \norm{y-x}_2 < \norm{x}_2/4} \Psi^\indi(x,y)\big( f(y) -f(x)\big) + R_1(x) \notag \\
      & = \sum_{y \, : \, \norm{y-x}_2 < \norm{x}_2/4} \Psi^\indi(x,y) \Big( 2(\beta/2-1)(y_1-x_1)\norm{x}_2^{-\beta}
        + 2(\beta/2-1)(y_2-x_2)\norm{x}_2^{-\beta} \notag \\[-2.5ex]
      & \hspace{11.5em} + \big( 2(\beta/2-1)\norm{x}_2^{-\beta}
        - 4(\beta/2-1)(\beta/2) x_1^2\norm{x}_2^{-\beta-2} \big) (y_1-x_1)^2 \notag \\
      & \hspace{11.5em} + \big( 2(\beta/2-1)\norm{x}_2^{-\beta}
        - 4(\beta/2-1)(\beta/2) x_2^2\norm{x}_2^{-\beta-2} \big) (y_2-x_2)^2 \notag \\
      & \hspace{11.5em} - 8 (\beta/2-1)(\beta/2) x_1 x_2 \norm{x}_2^{-\beta-2} (y_1-x_1)(y_2-x_2)
        + R_2(x,y) \Big) \notag \\
      & \qquad + R_1(x)
    \end{align}
    with $R_1(x) = \sum_{y \, : \, \norm{y-x}_2 \ge \norm{x}_2/4} \Psi^\indi(x,y) \big( f(y) -f(x)\big)$ and
    $R_2(x,y)$ being the remainder term when Taylor expanding $y \mapsto f(y)$ around $x$ to second order. 
    For example, using the Lagrange form of the remainder term, we have for some $\zeta = \zeta(x,y) \in \{ x + t(y-x) : 0 \le 1 \le t\}$
    \begin{align}
      \label{eq:R2 remainder}
      R_2(x,y)
      & = \frac{1}{6} \Big(
        \frac{\partial^3}{\partial x_1^3} f(\zeta) (y_1-x_1)^3
        + 3 \frac{\partial^3}{\partial x_1^2 \partial x_2} f(\zeta) (y_1-x_1)^2 (y_2-x_2) \notag \\
      & \qquad  \qquad
        + 3 \frac{\partial^3}{\partial x_1 \partial x_2^2} f(\zeta) (y_1-x_1) (y_2-x_2)^3
        + \frac{\partial^3}{\partial x_2^3} f(\zeta) (y_2-x_2)^3 \Big) .
    \end{align}
    
    We have
    \begin{align}
      \sum_{y \, : \, \norm{y-x}_2 < \norm{x}_2/4} \Psi^\indi(x,y) 2(\beta/2-1)\norm{x}^{-\beta/2}
      \big( (y_1-x_1) + (y_2-x_2)\big) = 0
    \end{align}
    by symmetry and we can bound
    \begin{align}
      |R_1(x)| \le \norm{f}_\infty \sum_{y \, : \, \norm{y-x}_2 \ge \norm{x}/4} \Psi^\indi(x,y) \le C M^{2-\beta}\norm{x}
      _2^{-\beta}
    \end{align}
    by the tail bounds on inter-regeneration times and increments, cf.\ \eqref{eq:regen_time_tail_bounds}.
    We have 
    \begin{align}
      \label{eq:R2 remainder bound}
      \left| \frac{\partial^3}{\partial x_1^3} f(\zeta) \right|
      + \left| \frac{\partial^3}{\partial x_1^2 \partial x_2} f(\zeta)\right|
      + \left| \frac{\partial^3}{\partial x_1 \partial x_2^2} f(\zeta) \right|
      + \left| \frac{\partial^3}{\partial x_2^3} f(\zeta)\right|
      \le \frac{C}{\norm{\zeta}_2^{\beta+1}}
      \quad \text{for } \norm{\zeta}_2 \ge 1,
    \end{align}
    thus using \eqref{eq:R2 remainder}, \eqref{eq:R2 remainder bound}
    and the fact that increments have (more than) three moments (cf.\ \eqref{eq:regen_time_tail_bounds})
    \begin{align}
      \sum_{y \, : \, \norm{y-x}_2 < \norm{x}_2/4} \Psi^\indi(x,y) |R_2(x,y)|
      & \le \sum_{y \, : \, \norm{y-x}_2 < \norm{x}_2/4} \Psi^\indi(x,y) \frac{|y_1-x_1|^3 + |y_2-x_2|^3}{\norm{\zeta(x,y)}_2^{\beta+1}}
        \notag \\
      & \le \frac{C}{\norm{x}_2^{\beta+1}} \sum_y \Psi^\indi(x,y) \big( |y_1-x_1|^3 + |y_2-x_2|^3 \big)
      \le \frac{C}{\norm{x}_2^{\beta+1}} .
    \end{align}
    
    Furthermore,
    \begin{align}
      & \sum_{y \, : \, \norm{y-x}_2 < \norm{x}_2/4} \Psi^\indi(x,y) \Big(
        \big( 2(\beta/2-1)\norm{x}_2^{-\beta}
        - 4(\beta/2-1)(\beta/2) x_1^2\norm{x}_2^{-\beta-2} \big) (y_1-x_1)^2 \notag \\[-2ex]
      & \hspace{10.5em} + \big( 2(\beta/2-1)\norm{x}_2^{-\beta}
        - 4(\beta/2-1)(\beta/2) x_2^2\norm{x}_2^{-\beta-2} \big) (y_2-x_2)^2 \notag \\
      & \hspace{10.5em} - 8 (\beta/2-1)(\beta/2) x_1 x_2 \norm{x}_2^{-\beta-2} (y_1-x_1)(y_2-x_2)
        \Big) \notag \\
      & = \frac{2\beta-2}{\norm{x}^{-\beta}}
        \sum_{y \, : \, \norm{y-x}_2 < \norm{x}_2/4} \Psi^\indi(x,y) \bigg(
        \Big( 1 - \beta\frac{x_1^2}{\norm{x}_2^2} \Big) (y_1-x_1)^2
        + \Big( 1 - \beta\frac{x_2^2}{\norm{x}_2^2} \Big) (y_2-x_2)^2 \notag \\[-1.5ex]
      & \hspace{26em} 
        -2 \beta\frac{x_1 x_2}{\norm{x}_2^2} (y_1-x_1)(y_2-x_2) \bigg) 
        \notag \\
      & = \frac{2\beta-2}{\norm{x}_2^{-\beta}}
        \Big( (1, 0) \widehat{C}^\indi_{\norm{x}_2} (1, 0)^T
        + (0, 1) \widehat{C}^\indi_{\norm{x}_2} (0, 1)^T \notag \\[-4.5ex]
      & \hspace{12em} -\beta (x_1/\norm{x}_2, x_2/\norm{x}_2) \widehat{C}^\indi_{\norm{x}_2} (x_1/\norm{x}_2, x_2/\norm{x}_2)^T \Big),
    \end{align}
    where
    \begin{align}
      \widehat{C}^\indi_{\norm{x}_2} = \sum_{z \, : \, \norm{z}_2 < {\norm{x}_2}/4} \Psi^\indi(0,z)
      \left( \begin{array}{cc} z_1^2 & z_1 z_2 \\ z_1 z_2 & z_2^2 \end{array}\right)
    \end{align}
    is the covariance matrix of an increment under $\Psi^\indi$ (restricted to jump size ${\norm{x}_2}/4$). 
    Let $\widehat{\lambda}_1 \ge \widehat{\lambda}_2 > 0$ be the eigenvalues of $\widehat{C}^\indi_{\norm{x}_2}$
    (by irreducibility, $\widehat{C}^\indi_{\norm{x}_2}$ must be invertible, at least when $M$ is large enough because $\norm{x}_2 \ge M$ by assumption). Then
    \begin{align}
      & \sup_{\substack{(x_1,x_2) \in \Z^2,\\ \norm{(x_1,x_2)}_2 \ge M}} \Big( (1, 0) \widehat{C}^\indi_{\norm{x}_2} (1, 0)^T
      + (0, 1) \widehat{C}^\indi_{\norm{x}_2} (0, 1)^T \notag \\[-3ex]
      & \hspace{15em} -\beta (x_1/\norm{x}_2, x_2/\norm{x}_2) \widehat{C}^\indi_{\norm{x}_2} (x_1/\norm{x}_2, x_2/\norm{x}_2)^T \Big)
      \notag \\
      & \quad 
        \le 2 \widehat{\lambda}_1 - \beta \widehat{\lambda}_2 < 0
        \label{eq:ev.covmat}
    \end{align}
    if $\beta>2\widehat{\lambda}_1/\widehat{\lambda}_2$, which can be achieved by suitably tuning the parameters (see discussion in Remark~\ref{rem:reasonably anisotropic}).

    In total, we find
    \begin{align}
      \sum_{y \in \Z^2} \Psi^\indi(x,y)\big( f(y) -f(x)\big) \le -c \norm{x}_2^{-\beta}
      \quad \text{for } \norm{x}_2 \ge M
    \end{align}
    with a constant $c>0$.

    Combined with the total
    variation bound on $\Psi^\joint-\Psi^\indi$ from \eqref{eq:TVdistance-joint-ind-1step in prop}
    and the fact that the tail bound \eqref{eq:regen_time_tail_bounds} holds for $\Psi^\joint$
    as well as for $\Psi^\indi$, this implies 
    \begin{align}
      \sum_{y \in \Z^2} \Psi^\joint(x,y)\big( f(y) -f(x)\big) 
      & \le \sum_{y \in \Z^2} \Psi^\indi(x,y)\big( f(y) -f(x)\big) \notag \\
      & \qquad
      + \norm{f}_\infty \sum_{y \in \Z^2} \abs{\Psi^\joint(x,y) - \Psi^\indi(x,y)}
      \le -c \norm{x}_2^{-\beta},
    \end{align}
    i.e.\ \eqref{eq:Poissonproblemannulus} holds true.
  \end{proof}

  \begin{remark}[Covariance of increments under $\Psi^\indi$ is not too anisotropic]
    \label{rem:reasonably anisotropic}
      Recall that ``tuning the parameters'' of the logistic branching random walk to conform with the assumptions of Theorem~\ref{thm:eta_survival} means choosing the competition parameters $\lambda_{xy}$ in \eqref{def:f}, \eqref{eq:law of eta in log branching} small. Then the local population density tends to be large with small relative fluctuations. Thus, regeneration will tend to occur fast and $\Psi^\indi(x,y)$ will be close to $p_{yx} = p_{x-y}$ where $p$ is the displacement kernel from parent to offspring in \eqref{eq:law of eta in log branching}. Since $p$ is by assumption irreducible with finite range, $\widehat{\lambda}_1$ and $\widehat{\lambda}_2$ will be uniformly bounded away from $0$ and from $\infty$ in the regime we consider, while we at the same time we can choose $\beta$ large. Thus, the condition on $\widehat{\lambda}_1$ and $\widehat{\lambda}_2$ from \eqref{eq:ev.covmat} can be fulfilled.
      
      Note also that if $(p_{xy})$ and $(\lambda_{xy})$ have additional symmetries, namely invariance under exchange of coordinates and plus mirror symmetry with respect to each coordinate individually, we have $\widehat{\lambda}_1 = \widehat{\lambda}_2$ and thus the condition in \eqref{eq:ev.covmat} is fulfilled for all $\beta>2$.
  \end{remark}
  
  \begin{remark}[The Green function of $2$-dimensional Brownian motion in an annulus]
    \label{rem:BMGreenfct.annulus} 
    For comparison, we recall here the answer corresponding to \eqref{eq:Poissonproblemannulus} for $2$-dimensional Brownian motion.
    The radius process 
    is in this case a $2$-dimensional Bessel process $(X_t)_{t \ge 0}$,
    solving $dX_t = dB_t +(2X_t)^{-1} dt$, i.e.\ the generator is given by
    $L f(x) = \frac12 \sigma^2(x) f''(x) + \mu(x) f'(x)$ 
    with $\sigma^2(x) \equiv 1$, $\mu(x)=1/(2x)$.

    We have
    \[
      s(\eta) =
      \exp\Big( - \int^\eta \frac{2\mu(\xi)}{\sigma^2(\xi)} \,d\xi \Big) =
      \exp\Big( - \int^\eta \frac1\xi \,d\xi \Big) = \exp(-\log \eta) = \frac1\eta,
    \]
    hence the scale function is
    \[
      S(x) = \int^x s(\eta) \, d\eta = \log x .
    \]
    Let $0 < a < b$, consider the process which is stopped upon leaving $(a,b)$.

    For $x \in [a,b]$ the Green function is
    \begin{equation}
      G_{a,b}(x,\xi) =
      \begin{cases}
        \displaystyle
        \frac{2(\log x - \log a)(\log b - \log \xi)}{\log b - \log a}\xi,
        & x \le \xi \le b, \\[2ex]
        \displaystyle
        \frac{2(\log b - \log x)(\log \xi - \log a)}{\log b - \log a}\xi,
        & a \le \xi \le x,
      \end{cases}
    \end{equation}
    see, e.g., 
    \cite[Rem.~3.3 in Ch.~15.3]{KarlinTaylor}.
    Thus for $\beta>2$ and $a < x < b$
    \begin{align}
      & \int_a^b G_{a,b}(x,\xi) \xi^{-\beta} \, d\xi \notag \\
      & = 2 \frac{\log b - \log x}{\log b - \log a} \int_a^x (\log \xi - \log a)\xi^{-(\beta-1)} \, d\xi
        + 2 \frac{\log x - \log a}{\log b - \log a} \int_x^b (\log b - \log \xi)\xi^{-(\beta-1)} \, d\xi \notag \\
      &
        = 2 \frac{\log b - \log x}{\log b - \log a} \bigg[ -\frac{(\beta-2)^{-1} + \log \xi}{(\beta-2) \xi^{\beta-2}}
        + \frac{\log a}{(\beta-2) \xi^{\beta-2}} \bigg]_{\xi=a}^{\xi=x} \notag \\
      & \quad + 2 \frac{\log x - \log a}{\log b - \log a} \bigg[
        - \frac{\log b}{(\beta-2) \xi^{\beta-2}} + \frac{(\beta-2)^{-1} + \log \xi}{(\beta-2) \xi^{\beta-2}}
        \bigg]_{\xi=x}^{\xi=b} \notag \\
      & = \frac{2}{\beta-2} \frac{\log b - \log x}{\log b - \log a} \bigg(
        -\frac{(\beta-2)^{-1} + \log x}{x^{\beta-2}}
        + \frac{\log a}{x^{\beta-2}}
        + \frac{(\beta-2)^{-1} + \log a}{a^{\beta-2}} - \frac{\log a}{a^{\beta-2}} \bigg)
        \notag \\
      & \quad + \frac{2}{\beta-2} \frac{\log x - \log a}{\log b - \log a} \bigg(
        - \frac{\log b}{b^{\beta-2}} + \frac{(\beta-2)^{-1} + \log b}{b^{\beta-2}}
        + \frac{\log b}{x^{\beta-2}} - \frac{(\beta-2)^{-1} + \log x}{x^{\beta-2}}
        \bigg) \notag \\
      & = \frac{2}{\beta-2} \frac{\log b - \log x}{\log b - \log a} \bigg(
        \frac{\log a - \log x}{x^{\beta-2}} - \frac{(\beta-2)^{-1}}{x^{\beta-2}} + \frac{(\beta-2)^{-1}}{a^{\beta-2}} \bigg)
        \notag \\
      & \quad + \frac{2}{\beta-2} \frac{\log x - \log a}{\log b - \log a} \bigg(
        \frac{\log b - \log x}{x^{\beta-2}} + \frac{(\beta-2)^{-1}}{b^{\beta-2}}
        - \frac{(\beta-2)^{-1}}{x^{\beta-2}}
        \bigg) \notag \\
      & = \frac{2}{(\beta-2)^2} \frac{\log b - \log x}{\log b - \log a} \big( a^{2-\beta} - x^{2-\beta} \big)
        + \frac{2}{(\beta-2)^2} \frac{\log x - \log a}{\log b - \log a} \big( b^{2-\beta} - x^{2-\beta} \big)
        \notag \\
      & = \frac{2}{(\beta-2)^2 (\log b - \log a)}
        \Big( \big(\log b - \log x \big)\big( a^{2-\beta} - x^{2-\beta} \big)
        + \big(\log x - \log a\big)\big( b^{2-\beta} - x^{2-\beta} \big) \Big)
        \notag \\
      & = \frac{2}{(\beta-2)^2 (\log b - \log a)}
        \Big( a^{2-\beta}\log b - x^{2-\beta}\log b - a^{2-\beta}\log x + x^{2-\beta}\log x \notag \\
      & \hspace{11.5em} + b^{2-\beta}\log x - x^{2-\beta}\log x - b^{2-\beta}\log a + x^{2-\beta}\log a
        \Big) \notag \\
      & = \frac{2}{(\beta-2)^2 (\log b - \log a)}
        \Big( -(\log b - \log a) x^{2-\beta} + \big( b^{2-\beta} - a^{2-\beta} \big) \log x
        + a^{2-\beta}\log b - b^{2-\beta}\log a \Big) \notag \\
      & =: f_{a,b}(x)
    \end{align}
    (note $\int \xi^{-c} \log \xi \, d\xi = -(c-1)^{-1} \xi^{1-c} \big((c-1)^{-1} + \log \xi\big)$ for $c>1$). It is straightforward to check that indeed $L f_{a,b}(x) = -x^{-\beta} $ for $a < x < b$.
  \end{remark}

\section{Coalescing probabilities}
\label{sect:Coalescing probabilities}
  
        We want to quantify the probabilities for the ancestral lineages to coalesce in the next step. For that we consider different cases of starting positions and possible positions of the ancestor. For this setting we condition on the environmental process $\eta$, which gives us the number of individuals on each site. As we recall, given $\eta$ we can define the dynamics of the ancestral lineages. To that we now want to add the probabilities for coalescence of two ancestral lineages.

        To that end let $Y_{n,i}(x,y)$ be the number of children of the $i$-th individual from site $x$ at time $n$ that migrate to site $y$. By definition of the $Y$'s we obtain
        \begin{equation}
                \label{eq:eta(n+1)_as_a_sum_of_Ys}
                \eta_{n+1}(y) = \sum_{x\in B_{R_X}(y)} \sum_{i=1}^{\eta_n(x)} Y_{n,i}(x,y).
        \end{equation}

        Now let $y_1,y_2,x\in\bZ^2$ and $y_1\neq y_2$. Say we knew the values of all $Y$'s, then the probability for an ancestral lineage at position $y_1$ and another at position $y_2$ at time $n+1$ to coalesce at time $n$ at position $x$ is given by

        \begin{align*}
                \sum_{i=1}^{\eta_n(x)} \frac{Y_{n,i}(x,y_1)}{\eta_{n+1}(y_1)} \cdot \frac{Y_{n,i}(x,y_2)}{\eta_{n+1}(y_2)},
        \end{align*}
        where, if the ancestral lineage is at some site, we choose the corresponding individual uniformly among all available. Since the information provided by the environment is only the values of $\eta$, we need to calculate the conditional expectation of the above expression, conditioned on $\eta$. More precisely, it is sufficient to condition on $\eta_n$ and $\eta_{n+1}$.

        \begin{align*}
                \bE&\Big[ \sum_{i=1}^{\eta_n(x)} \frac{Y_{n,i}(x,y_1)}{\eta_{n+1}(y_1)} \cdot \frac{Y_{n,i}(x,y_2)}{\eta_{n+1}(y_2)}\,\vert\, \eta_n,\eta_{n+1} \Big]\\
                &= \frac{1}{\eta_{n+1}(y_1)\eta_{n+1}(y_2)} \sum_{i=1}^{\eta_n(x)} \bE\big[ Y_{n,i}(x,y_1)Y_{n,i}(x,y_2)\,\vert\, \eta_n,\eta_{n+1} \big]\\
                &= \frac{1}{\eta_{n+1}(y_1)\eta_{n+1}(y_2)} \sum_{i=1}^{\eta_n(x)} \bE\big[ Y_{n,i}(x,y_1)\,\vert\, \eta_n,\eta_{n+1} \big]\bE\big[ Y_{n,i}(x,y_2)\,\vert\, \eta_n,\eta_{n+1} \big],
        \end{align*}
        where in the last line we used that $Y_{n,i}(x,y_1)$ and $Y_{n,i}(x,y_2)$ are, conditioned on $\eta_n$, in fact independent Poisson random variables with parameters $f(x;\eta_n)p_{x,y_1}/\eta_n(x)$ and $f(x;\eta_n)p_{x,y_2}/\eta_n(x)$. Note that also $\eta_{n+1}(y) \sim \text{Poi}(\sum_z p_{z,y}f(z;\eta_n))$ and thus, by \eqref{eq:eta(n+1)_as_a_sum_of_Ys}, conditioned on $\eta_{n+1}(y)$ the collection of $Y_{n,i}(x,y)$'s has a multinomial distribution with parameters $\eta_{n+1}(y)$ and
        \begin{equation*}
                \frac{\frac{f(x;\eta_n)}{\eta_n(x)}p_{x,y}}{\sum_z f(z;\eta_n)p_{z,y}}.
        \end{equation*}
        Therefore $Y_{n,i}(x,y_1)$ in the above conditioned expectation is binomial random variable and
        \begin{equation}
                \label{eq:condition_expectation_single_Y}
                \bE\big[ Y_{n,i}(x,y_1)\,\vert\, \eta_n,\eta_{n+1} \big] = \eta_{n+1}(y_1) \frac{\frac{f(x;\eta_n)}{\eta_n(x)}p_{x,y}}{\sum_z f(z;\eta_n)p_{z,y}}.
        \end{equation}
        It follows that the probability for coalescence on the site $x$ at time $n$, starting from $y_1$ and $y_2$ at time $n+1$ is given by
        \begin{align*}
                p^{(2)}_{\eta,\text{coal}}\big(n;(y_1,y_2),(x,x)\big)
                &= \frac{\eta_n(x)}{\eta_{n+1}(y_1)\eta_{n+1}(y_2)} \frac{\eta_{n+1}(y_1)\frac{f(x;\eta_n)}{\eta_n(x)}p_{x,y_1}\cdot\eta_{n+1}(y_2)\frac{f(x;\eta_n)}{\eta_n(x)}p_{x,y_2}}{\sum_z f(z;\eta_n)p_{z,y_1}\cdot\sum_{z'}f(z';\eta_n)p_{z',y_2}}\\
                &= \frac{f(x;\eta_n)p_{x,y_1}\cdot f(x;\eta_n)p_{x,y_2}}{\eta_n(x)\sum_z f(z;\eta_n)p_{z,y_1}\cdot\sum_{z'}f(z';\eta_n)p_{z',y_2}}.
        \end{align*}
        Next we study the probability that the two ancestral lineages meet on the same site but do not coalesce. Again we can consider what this probability would look like if we knew the values of all $Y$'s. Here we get, again for $y_1\neq y_2$,
        \begin{equation*}
                \sum_{i=1}^{\eta_n(x)}\frac{Y_{n,i}(x,y_1)}{\eta_{n+1}(y_1)} \sum_{j\neq i}^{\eta_n(x)} \frac{Y_{n,j}(x,y_2)}{\eta_{n+1}(y_2)}.
        \end{equation*}
        By similar arguments as above we obtain
        \begin{align*}
                \bE&\Big[ \sum_{i=1}^{\eta_n(x)}\frac{Y_{n,i}(x,y_1)}{\eta_{n+1}(y_1)} \sum_{j\neq i}^{\eta_n(x)} \frac{Y_{n,j}(x,y_2)}{\eta_{n+1}(y_2)}\,\vert\, \eta_n,\eta_{n+1} \Big]\\
                &=\eta_n(x)(\eta_n(x)-1)\frac{\frac{f(x;\eta_n)}{\eta_n(x)}p_{x,y_1} \cdot\frac{f(x;\eta_n)}{\eta_n(x)}p_{x,y_2}}{\sum_z f(z;\eta_n)p_{z,y_1} \cdot \sum_{z'} f(z';\eta_n)p_{z',y_2}}.
        \end{align*}
        And thus
        \begin{equation}
                p^{(2)}_{\eta,\text{non-coal}}\big(n; (y_1,y_2),(x,x) \big) = \Big(1-\frac{1}{\eta_n(x)}\Big)\frac{f(x;\eta_n)p_{x,y_1}f(x;\eta_n)p_{x,y_2}}{\sum_z f(z;\eta_n)p_{z,y_1} \cdot \sum_{z'} f(z';\eta_n)p_{z',y_2}}.
        \end{equation}
        Next we consider the two lineages starting from the same site and separating. Let $x_1,x_2\in\bZ^2$ and $x_1\neq x_2$. By similar considerations as above we start with
        \begin{align*}
                \bE&\Big[ \sum_{i=1}^{\eta_n(x_1)} \sum_{j=1}^{\eta_n(x_2)} \frac{Y_{n,i}(x_1,y) Y_{n,j}(x_2,y)}{\eta_{n+1}(y) (\eta_{n+1}(y)-1)}\,\vert\, \eta_n,\eta_{n+1} \Big]\\
                &=\frac{1}{\eta_{n+1}(y)(\eta_{n+1}(y)-1)}\sum_{i=1}^{\eta_n(x_1)} \sum_{j=1}^{\eta_n(x_2)} \eta_{n+1}(y)(\eta_{n+1}(y)-1)\frac{\frac{f(x_1;\eta_n)}{\eta_n(x_1)}p_{x_1,y}}{\sum_z f(z;\eta_n)p_{z,y}} \cdot\frac{\frac{f(x_2;\eta_n)}{\eta_n(x_2)}p_{x_2,y}}{\sum_{z'}f(z';\eta_n)p_{z',y}}\\
                &=\frac{f(x_1;\eta_n)p_{x_1,y} f(x_2;\eta_n)p_{x_2,y}}{(\sum_z f(z;\eta_n)p_{z,y})^2}
        \end{align*}
        and thus, if the lineages are yet to coalesce
        \begin{equation}
                p^{(2)}_{\eta}\big(n; (y,y),(x_1,x_2) \big) = \frac{f(x_1;\eta_n)p_{x_1,y} f(x_2;\eta_n)p_{x_2,y}}{(\sum_z f(z;\eta_n)p_{z,y})^2}.
        \end{equation}
        For the case of different start and end positions we obtain
        \begin{align*}
                p^{(2)}_{\eta}\big(n; (y_1,y_2),(x_1,x_2) \big) &= \bE\Big[ \sum_{i=1}^{\eta_n(x_1)} \sum_{j=1}^{\eta_n(x_2)} \frac{Y_{n,i}(x_1,y_1)}{\eta_{n+1}(y_1)} \cdot \frac{Y_{n,j}(x_2,y_2)}{\eta_{n+1}(y_2)}\,\vert\, \eta_n,\eta_{n+1} \Big]\\
                &= \frac{f(x_1,\eta_n)p_{x_1,y_1} \cdot f(x_2;\eta_n)p_{x_2,y_2}}{\sum_z f(z;\eta_n)p_{z,y_1} \cdot \sum_{z'}f(z';\eta_n)p_{z',y_2}}.
        \end{align*}
        That leaves the calculations for starting on the same position and moving to the same position with and without coalescing.

        \begin{align*}
                p^{(2)}_{\eta,\text{coal}}\big(n;(y,y),(x,x) \big) &=\bE\Big[ \sum_{i=1}^{\eta_n(x)} \frac{Y_{n,i}(x,y)(Y_{n,i}(x,y)-1)}{\eta_{n+1}(y)(\eta_{n+1}(y)-1)}\,\vert\, \eta_n,\eta_{n+1} \Big]\\
                &=\frac{1}{\eta_{n+1}(y)(\eta_{n+1}(y)-1)} \sum_{i=1}^{\eta_n(x)} \eta_{n+1}(y)(\eta_{n+1}(y)-1)\bigg(\frac{\frac{f(x;\eta_n)}{\eta_n(x)}p_{x,y}}{\sum_z f(z;\eta_n)p_{z,y}}\bigg)^2\\
                &= \frac{1}{\eta_n(x)} \bigg(\frac{f(x;\eta_n)p_{x,y}}{\sum_z f(z;\eta_n)p_{z,y}} \bigg)^2.
        \end{align*}
        This leaves us with the last case
        \begin{align*}
                p^{(2)}_{\eta,\text{non-coal}}\big(n; (y,y),(x,x) \big)&=\bE\Big[ \sum_{i=1}^{\eta_n(x)}\sum_{j\neq i}^{\eta_n(x)} \frac{Y_{n,i}(x,y)Y_{n,j}(x,y)}{\eta_{n+1}(y)(\eta_{n+1}(y)-1)}\,\vert\, \eta_n,\eta_{n+1} \Big]\\
                &=\frac{1}{\eta_{n+1}(y)(\eta_{n+1}(y)-1)} \sum_{i=1}^{\eta_n(x)}\sum_{j\neq i}^{\eta_n(x)} \eta_{n+1}(y)(\eta_{n+1}(y)-1)\bigg(\frac{\frac{f(x;\eta_n)}{\eta_n(x)}p_{x,y}}{\sum_z f(z;\eta_n)p_{z,y}}\bigg)^2\\
                &=\frac{\eta_n(x)(\eta_n(x)-1)}{(\eta_n(x))^2} \bigg(\frac{\frac{f(x;\eta_n)}{\eta_n(x)}p_{x,y}}{\sum_z f(z;\eta_n)p_{z,y}}\bigg)^2\\
                &=\Big(1-\frac{1}{\eta_n(x)}\Big)\bigg(\frac{\frac{f(x;\eta_n)}{\eta_n(x)}p_{x,y}}{\sum_z f(z;\eta_n)p_{z,y}}\bigg)^2,
        \end{align*}
        which concludes all cases. To summarize the calculations, let $x,y,x_1,x_2,y_1,y_2\in\bZ^d$ be pairwise different. Then we have the following cases for a transition of two ancestral lineages:

                \begin{equation}
                \label{eq:D5}
                        p^{(2)}_{\eta,\text{coal}}\big(n;(y_1,y_2),(x,x)\big)
                        = \frac{f(x;\eta_n)p_{x,y_1}\cdot f(x;\eta_n)p_{x,y_2}}{\eta_n(x)\sum_z f(z;\eta_n)p_{z,y_1}\cdot\sum_{z'}f(z';\eta_n)p_{z',y_2}}
                \end{equation}

                \begin{equation}
                \label{eq:D6}
                        p^{(2)}_{\eta,\text{non-coal}}\big(n; (y_1,y_2),(x,x) \big) = \Big(1-\frac{1}{\eta_n(x)}\Big)\frac{f(x;\eta_n)p_{x,y_1}f(x;\eta_n)p_{x,y_2}}{\sum_z f(z;\eta_n)p_{z,y_1} \cdot \sum_{z'} f(z';\eta_n)p_{z',y_2}}
                \end{equation}

                \begin{equation}
                \label{eq:D7}
                        p^{(2)}_{\eta}\big(n; (y,y),(x_1,x_2) \big) = \frac{f(x_1;\eta_n)p_{x_1,y} f(x_2;\eta_n)p_{x_2,y}}{(\sum_z f(z;\eta_n)p_{z,y})^2}
                \end{equation}

                \begin{equation}
                \label{eq:D8}
                        p^{(2)}_{\eta}\big(n; (y_1,y_2),(x_1,x_2) \big)= \frac{f(x_1,\eta_n)p_{x_1,y_1} \cdot f(x_2;\eta_n)p_{x_2,y_2}}{\sum_z f(z;\eta_n)p_{z,y_1} \cdot \sum_{z'}f(z';\eta_n)p_{z',y_2}}
                \end{equation}

                \begin{equation}
                \label{eq:D9}
                        p^{(2)}_{\eta,\text{coal}}\big(n;(y,y),(x,x) \big) = \frac{1}{\eta_n(x)} \bigg(\frac{f(x;\eta_n)p_{x,y}}{\sum_z f(z;\eta_n)p_{z,y}} \bigg)^2
                \end{equation}

                \begin{equation}
                \label{eq:D10}
                        p^{(2)}_{\eta,\text{non-coal}}\big(n; (y,y),(x,x) \big)=\Big(1-\frac{1}{\eta_n(x)}\Big)\bigg(\frac{\frac{f(x;\eta_n)}{\eta_n(x)}p_{x,y}}{\sum_z f(z;\eta_n)p_{z,y}}\bigg)^2.
                \end{equation}

        Note that we only need to consider the chance of coalescence if the two lineages meet on the same site, which is why we split those cases into a probability for coalescence and non-coalescence.
\end{appendix}

\subsection*{Acknowledgements}
M.B.\ and A.D.\ would like to thank Alison Etheridge for many inspiring discussions over the years.

%\bibliographystyle{amsalpha}
%\bibliography{Pct2dlbrw}

\newcommand{\etalchar}[1]{$^{#1}$}
\providecommand{\bysame}{\leavevmode\hbox to3em{\hrulefill}\thinspace}
\providecommand{\MR}{\relax\ifhmode\unskip\space\fi MR }
% \MRhref is called by the amsart/book/proc definition of \MR.
\providecommand{\MRhref}[2]{%
	\href{http://www.ams.org/mathscinet-getitem?mr=#1}{#2}
}
\providecommand{\href}[2]{#2}

\end{document}